\documentclass[11pt]{article}

\usepackage[
            CJKbookmarks=true,
            bookmarksnumbered=true,
            bookmarksopen=true,
            colorlinks=true,
            citecolor=red,
            linkcolor=blue,
            anchorcolor=red,
            urlcolor=blue,
            pdfauthor={Kiryung Lee, Rakshith Sharma Srinivasa, Marius Junge, Justin Romberg},
            pdfstartview=FitH,
            ]{hyperref}

\usepackage{amsmath,amsxtra,amssymb,amsthm,amsfonts}
\usepackage{mathabx}
\usepackage[usenames]{color}
\usepackage{umoline}
\usepackage[all]{xy}
\usepackage{bm}
\usepackage{bbm}
\usepackage{tikz-cd}
\usepackage{epstopdf}
\usepackage{multirow}
\usepackage{graphicx}
\usepackage{geometry}
\usepackage[caption=false]{subfig}
\usepackage[numbers,sort&compress]{natbib}
\usepackage{algorithm}
\usepackage{algorithmic}

\newtheorem{lemma}{Lemma}[section]
\newtheorem{prop}[lemma]{Proposition}
\newtheorem{theorem}[lemma]{Theorem}
\newtheorem{cor}[lemma]{Corollary}

\newtheorem{rem}[lemma]{Remark}

\newcommand{\re}{\begin{rem}\rm}
  \newcommand{\mar}{\end{rem}}

\newtheorem{defi}[lemma]{Definition}

\newcommand{\xspace}{\hbox{\kern-2.5pt}}
\newcommand\tnorm[1]{\left\vert\xspace\left\vert\xspace\left\vert\mskip2mu
#1\mskip2mu \right\vert\xspace\right\vert\xspace\right\vert}

\newcommand{\R}{\mathbb{R}}
\newcommand{\norm}[1]{\left \|#1 \right \|}
\newcommand{\E}{\mathbb{E}}

\begin{document}

\title{Approximately low-rank recovery from noisy and local measurements by convex program}

\author{Kiryung Lee, Rakshith Sharma Srinivasa, Marius Junge, and Justin Romberg\thanks{K. Lee is with the Department of Electrical and Computer Engineering at the Ohio State University, Columbus, OH 43210. R.S. Srinivasa and J. Romberg are with the School of Electrical and Computer Engineering at Georgia Institute of Technology, Atlanta, GA 30308. M. Junge is with the Department of Mathematics at University of Illinois at Urbana-Champaign, Urbana, IL 61801. Preliminary results were presented in part at the 13th International conference on Sampling Theory and Applications \cite{lee2019entropy} and at the 33rd Conference on Neural Information Processing Systems \cite{srinivasa2019decentralized}.}}

\maketitle

\begin{abstract}
Low-rank matrix models have been universally useful for numerous applications, from classical system identification to more modern matrix completion in signal processing and statistics. The nuclear norm has been employed as a convex surrogate of the low-rankness since it induces a low-rank solution to inverse problems. While the nuclear norm for low rankness has an excellent analogy with the $\ell_1$ norm for sparsity through the singular value decomposition, other matrix norms also induce low-rankness. Particularly as one interprets a matrix as a linear operator between Banach spaces, various tensor product norms generalize the role of the nuclear norm. We provide a tensor-norm-constrained estimator for the recovery of approximately low-rank matrices from local measurements corrupted with noise. A tensor-norm regularizer is designed to adapt to the local structure. We derive statistical analysis of the estimator over matrix completion and decentralized sketching by applying Maurey's empirical method to tensor products of Banach spaces. The estimator provides a near-optimal error bound in a minimax sense and admits a polynomial-time algorithm for these applications.
\end{abstract}

\section{Introduction}

We consider the estimation of an approximately low-rank matrix $M_0 \in \mathbb{R}^{d_1 \times d_2}$ from noisy and \emph{local} linear measurements 
\begin{equation}
	\label{eq:meas_mdl}
	y_k = \mathrm{tr}(A_k^\top M_0) + \eta_k, \quad k=1,\dots,n,
\end{equation}
where the $A_k$'s are fixed measurement matrices and the $\eta_k$ are additive noise.  The $A_k$'s are local in the sense that they have very few non-zero entries; each measurement $y_k$ depends only on a small part of the unknown $M_0$.  We are particularly interested in the cases where $A_k$'s have only a single non-zero column (which we refer to as column-wise sketching) and where $A_k$ has a single non-zero entry (as in the well-known matrix completion problem).

Given the $y_k$ in \eqref{eq:meas_mdl}, we will estimate $M_0$ by solving
\begin{equation}
	\label{eq:gen_mat_lasso}
	\begin{array}{ll}
		\displaystyle \mathop{\text{minimize}}_M & \|y - \mathcal{A}(M)\|_2^2 \\
		\text{subject to} & \tnorm{M} \leq \alpha,
	\end{array}
\end{equation}
where $\tnorm{\cdot}$ is a matrix norm and we have collected the measurements into $y = [y_1 ; \dots ; y_n]$ and the collective action of the measurement matrix $A_k$ into the operator $\mathcal{A}: \mathbb{R}^{d_1 \times d_2} \to \mathbb{R}^n$. We can interpret \eqref{eq:gen_mat_lasso} as a generalization of the classic LASSO estimator \citep{tibshirani1996regression} where we will choose the norm $\tnorm{\cdot}$ based on properties of the $A_k$.  We will show that careful choice of $\tnorm{\cdot}$ results in minimax near-optimal error bounds for both column-wise sketching and matrix completion. 

\subsection{Regularizer for local measurement operator}

We present a unifying framework to design the regularizer $\tnorm{\cdot}$ according to the structure in the measurement operator $\mathcal{A}$ for three different types of measurement models.

\vspace{.1in}
\noindent
\textbf{Column-wise sketching.}
In this acquisition model, the columns of $M_0$ are sensed independently.  Each $A_k$ is an outer product of a random Gaussian vector and a standard basis vector: $A_k = \xi_k e_{j_k}^\top$, where $j_k \in [d_2]$ and $e_j$ is the $j$th column of the identity matrix. This model arises in multivariate regression with applications in the analysis of medical imaging and financial data \cite{negahban2011estimation}.
In general, there will be multiple measurements for each column ($j_k$ will be the same for multiple $k$s). Therefore, the net effect is applying a different random matrix to each column of $M_0$.  

The natural question is whether we can get a good estimate of a rank-$r$ matrix $M_0$ using $\sim r$ measurements per column.  If the column space of $M_0$ is known, then this is clearly possible, either by designing the $\xi_k$ to live in this column space or even taking them to be random.  But in the (typical) case where we do not know the column space, we can see that there are some rank-$r$ matrices we will be able to estimate better than others. In particular, if all of the energy in $M_0$ is concentrated in a small number of columns, then we will not benefit from the low-rank nature of the matrix as much as if the energy were distributed more evenly.  In the extreme case where $M_0$ has a single non-zero column, we will need to use $\sim d_1$ measurements per column to recover the matrix (but most of these will be wasted on sensing columns that are zero) no matter what recovery scheme we use.

We expect that we will be able to better estimate matrices with roughly even energy distribution throughout their columns and so we will design our regularizer to favor such matrices. A natural way to favor $d_1\times d_2$ matrices that have roughly equal energy in the columns is through the  $\ell_1^{d_2}\rightarrow\ell_2^{d_1}$ operator norm
\begin{equation}
	\label{eq:12norm}
	\|M\|_{1\rightarrow 2} = \max_{j\in[d_2]} \|Me_j\|_2,
\end{equation}
which directly penalizes the maximum norm of the individual columns of $M$, thus encouraging these norms to be the same size.  
Alternatively, the regularizer in \eqref{eq:12norm} is derived through the ``maximum'' correlation of the matrix variable with all possible measurement matrices. 
Since $\xi_k$ has a continuous distribution, the maximization over $\xi_k$ can be simplified as follows: 
Recall that the maximum of $m$ i.i.d. subgaussian random variables is upper-bounded by the standard deviation within a logarithmic factor of $m$ with high probability. 
From this observation, we consider the maximum correlation defined by
\begin{equation}
\label{eq:max_cor_ds}
\max_{j \in [d_2]} \sqrt{\mathbb{E} (\xi^\top M e_j)^2} 
= \max_{j \in [d_2]} \|M e_j\|_2,
\end{equation}
where the entries of $\xi \in \mathbb{R}^{d_1}$ are i.i.d. following a subgaussian distribution, $\mathbb{E} \xi = 0$, and $\mathbb{E} \xi \xi^\top = I_{d_1}$. 
Note that the right-hand side of \eqref{eq:max_cor_ds} coincides with the norm in \eqref{eq:12norm}.

While the max column norm does not by itself favor matrices that are low rank, there exists a variation on the nuclear norm that complements $\|\cdot\|_{1\rightarrow 2}$.  We call this the \emph{mixed norm}:
\begin{equation}
	\label{eq:def_mixed_norm}
	\norm{M}_\mathrm{mixed} := \inf_{U,V : U V^\top = M} \norm{U}_\mathrm{F} \|V^\top\|_{1\to 2}. 
\end{equation}
Compared to the nuclear norm, which minimizes the product of the Frobenius norms of candidate factors $U,V$ (see \eqref{eq:nuc_norm_by_opt_factorization} below), the mixed norm uses the (isotropic) Frobenius norm for the left factor and $\|\cdot\|_{1\rightarrow2}$ for the right factor. 
As we will discuss in Section~\ref{sec:prelim} below, the mixed norm in \eqref{eq:def_mixed_norm} is what is known as the \emph{projective tensor norm} when viewing $M$ as a linear map from $\ell_1^{d_2}$ into $\ell_2^{d_1}$.  It favors low-rank matrices in a way that complements $\|\cdot\|_{1\rightarrow 2}$ as follows: while it is always true that $\|M\|_{1\rightarrow 2}\leq \|M\|_\mathrm{mixed}$, if $M$ is rank-$r$ we have the additional upper bound
\[
\norm{M}_{1 \to 2} \leq \norm{M}_\mathrm{mixed} \leq \sqrt{r} \norm{M}_{1\to 2}.
\]
In light of this, our regularizer for column-wise sketching will be
\begin{equation}
	\label{eq:tnorm_mixed}
	\tnorm{M} = \max\left( \norm{M}_{1\to 2}, \frac{\norm{M}_\mathrm{mixed}}{\sqrt{r}} \right).
\end{equation}
The intuition here is that if $M$ is rank-$r$, then the first term in \eqref{eq:tnorm_mixed} will be ``active'': among all rank-$r$ matrices, \eqref{eq:tnorm_mixed} favors those with smaller max column norms.  The convex set $\mathcal{C} = \{M \in \mathbb{R}^{d_1 \times d_2} : \tnorm{M} \leq \alpha\}$ can be interpreted as a relaxation of the set of all rank-$r$ matrices that have $\|M\|_{1\rightarrow 2}\leq \alpha$.

\vspace{.1in}
\noindent
\textbf{Matrix completion.}
In the matrix completion problem, we observe a subset of the entries of a low-rank matrix; the measurement matrices  $A_k$ are proportional to $\check{e}_{i_k} e_{j_k}^\top$, where $i_k \in [d_1]$, $j_k \in [d_2]$, and $\check{e}_i$ (resp. $e_i$) denotes the $i$th column of the identity matrix of size $d_1$ (resp. $d_2$).  We will consider the case where the indexes $i_k,j_k$ are chosen uniformly at random (and hence are independent of the matrix being sampled).  
It is again clear (and has been noted many times in the literature, e.g. \citep{candes2012exact}) that this type of measurement will not work equally well for all rank-$r$ matrices.  It will be most effective for matrices that are diffuse, meaning that they have their energy spread more or less evenly across their entries.

The natural way to favor matrices that are diffuse is by penalizing the maximum entry in the matrix,
\begin{equation}
	\label{eq:linfnorm}
	\|M\|_\infty = \max_{i \in [d_1], j \in [d_2]} | \check{e}_i^\top M e_j|
\end{equation}
Note that the right-hand side of \eqref{eq:linfnorm} corresponds to the maximum correlation of the matrix variable with the measurement matrices.

As with $\|\cdot\|_{1\rightarrow 2}$ norm for column-sketching, the $\|\cdot\|_\infty$ norm by itself does nothing to favor low-rank matrices.  There is, however, another variation on the nuclear norm, introduced by \citep{linial2007complexity} as the \emph{max norm}, that complements $\|\cdot\|_\infty$:
\begin{equation}
	\label{eq:def_max_norm}
	\norm{M}_\mathrm{max} := \inf_{U,V : U V^\top = M} \|U^\top\|_{1 \to 2} \|V^\top\|_{1\to 2}.
\end{equation}
Among matrices with fixed $\|\cdot\|_\infty$ norm, low-rank matrices will have a smaller max norm than general matrices.  In particular, rank-$r$ matrices obey
\[
	\|M\|_\infty\leq \|M\|_\mathrm{max}\leq \sqrt{r}\|M\|_\infty.
\]
Thus as a regularizer for matrix completion, we use
\begin{equation}
	\label{eq:tnorm_max}
	\tnorm{M} = \max\left( \norm{M}_{\infty}, \frac{\norm{M}_\mathrm{max}}{\sqrt{r}} \right).
\end{equation}
Again, the intuition here is that if the matrix is rank-$r$, then the $\|\cdot\|_\infty$ term above will be active; among all rank-$r$ matrices, \eqref{eq:tnorm_max} favors matrices that are diffuse.  Low-rank matrix completion with the penalty \eqref{eq:tnorm_max} was shown to be near-optimal in \cite{cai2016matrix}; in Section~\ref{subsec:main} below we show how this problem can be analyzed (again obtaining near-optimal results) using the same framework we develop for the column-sketching problem above.

\vspace{.1in}
\noindent
\textbf{Non-local random sketching.}  
For this measurement model, we observe ``non-local'' measurements of the form $A_1, \dots, A_n$ are independent copies of a random matrix of i.i.d. entries following $\mathcal{N}(0,1)$. 
In this case, the $A_k$ touches each part of $M_0$ equally, so it will not matter how the energy is distributed across the entries. 
Similar to the previous examples, we consider a norm derived from the maximum correlation by
\[
\max_{k \in [n]} \sqrt{\E \langle A_k, M \rangle^2} = \norm{M}_\mathrm{F}.
\]
We couple it with the standard nuclear norm 
\begin{equation}
\label{eq:nuc_norm_by_opt_factorization}
	\|M\|_* =  \inf_{U,V : U V^\top = M} \|U^\top\|_{\mathrm{F}} \|V^\top\|_{\mathrm{F}}
\end{equation}
so that the norm regularizer in \eqref{eq:gen_mat_lasso} becomes 
\begin{equation}
	\label{eq:tnorm_nl}
	\tnorm{M} = \max\left( \norm{M}_{\mathrm{F}}, \frac{\norm{M}_*}{\sqrt{r}} \right).
\end{equation}
Since $\|M\|_\mathrm{F} \leq \|M\|_* \leq \sqrt{r} \|M\|_\mathrm{F}$, this regularizer induces low-rankness. 
However, unlike the previous examples, the choice of the Frobenius norm in \eqref{eq:tnorm_nl} implies that we do not need to differentiate within the set of matrices that have the same rank.
Below we show that the analysis framework also applies to this example in a unifying way. 
If the noise terms $\eta_1,\dots,\eta_n$ are i.i.d. $\mathcal{N}(0,\sigma^2)$, then Proposition~\ref{prop:deterministic_errbnd} implies that the following result holds with probability $1-(d_1+d_2)^{-1}$: For all $M_0$ satisfying $\norm{M_0}_\mathrm{F} \leq \alpha$ and $\norm{M_0}_* \leq \sqrt{r} \alpha$, the estimate $\Phi$ by \eqref{eq:gen_mat_lasso} satisfies 
\begin{equation}
\label{eq:samp_comp_nonlocal}
\norm{\Phi-M_0}_\mathrm{F}^2 
\lesssim \alpha^2 \max\left(1, \frac{\sigma \sqrt{\log(d_1+d_2)}}{\alpha} \right) \cdot \sqrt{\frac{r(d_1+d_2)\log^6(d_1d_2)}{n}}
\end{equation}
provided $n \geq r(d_1+d_2) \log^6 (d_1d_2)$. (The proof is provided in Appendix~\ref{sec:nonlocal}.)
Let us consider a subset of the model where $M_0$ satisfies $\norm{M_0}_* \leq \sqrt{r} \norm{M_0}$, which corresponds to an approximately low-rank case. 
In this scenario, the error bound in \eqref{eq:samp_comp_nonlocal} is weaker than the analogous result obtained by simply constraining $\tnorm{M} = \norm{M}_*$ \citep[Eq. (II.7)]{candes2011tight} in the following senses: First, the error bound in \eqref{eq:samp_comp_nonlocal} does not vanish in the limit $\sigma \to 0$, which restricts its application to only high noise regimes. 
Furthermore, the error bound by \citet{candes2011tight} decays faster without the square root and any logarithmic factor.
The weakness of the result in \eqref{eq:samp_comp_nonlocal} is due to the extension of the model from $\{M: \norm{M}_* \leq \sqrt{r} \norm{M}_\mathrm{F}\}$ to the convex set $\{M: \norm{M}_\mathrm{F} \leq \alpha, \norm{M}_* \leq \sqrt{r}\alpha\}$. 

\subsection{A geometric characterization of estimation error bound}
We present a generalized analysis framework to derive recovery guarantees for the optimization program in \eqref{eq:gen_mat_lasso}. Our main goal is to bound the error between the estimate by \eqref{eq:gen_mat_lasso} and the ground truth $M_0$. 
In the following proposition, we provide a deterministic upper bound on the estimation error conditioned on the event when the measurement operator $\mathcal{A}$ satisfies specific properties. 
The proof of the proposition follows the technique used by \citet{cai2016matrix} to derive their error bound. We provide the proof in Appendix~\ref{sec:proof:prop:deterministic_errbnd} for completeness. 

\begin{prop}
	\label{prop:deterministic_errbnd}
	Let $\eta_1,\dots,\eta_n$ in \eqref{eq:meas_mdl} be i.i.d. drawn from $\mathcal{N}(0,\sigma^2)$. 
	Suppose that there exist parameters $\theta, \Gamma >0$ determined by $d_1,d_2,n,\alpha,r$ so that 
	\begin{equation}
		\label{eq:qub}
		\sup_{\vert\vert\vert M\vert\vert\vert \leq 1} \left| \frac{1}{n} \sum_{k=1}^n \mathrm{tr}(A_k^\top M )^2 - \| M \|_\mathrm{F}^2 \right| \leq \theta
	\end{equation}
	and
	\begin{equation}
		\label{eq:gaussian_complexity}
		\E_{(g_k)} \left\vert\xspace\left\vert\xspace\left\vert \sum_{k=1}^n g_{k} A_{k} \right\vert\xspace\right\vert\xspace\right\vert_{*} \leq \Gamma,
	\end{equation} 
	where $\tnorm{\cdot}_*$ denotes the dual norm of $\tnorm{\cdot}$ and $(g_k)_{k=1}^n$ is a sequence of i.i.d. standard Gaussian random variables.
	Let $R=\sup_{\vert\vert\vert M\vert\vert\vert\leq 1} \|M\|_{\mathrm{F}}$ be the radius of the $\tnorm{\cdot}$ unit ball in the Frobenius norm.
	Then the following statement holds with probability at least $1-\zeta$: for every $M_0 \in \{M~:~\tnorm{M}\leq\alpha\}$, the estimate by \eqref{eq:gen_mat_lasso} from the noisy measurements in \eqref{eq:meas_mdl} satisfies 
	\begin{equation}
		\label{eq:generic_upper_bound}
		\| \Phi - M_0\|_\mathrm{F}^2 \leq 4 \alpha^2 \theta + \frac{4\alpha\sigma  \Gamma}{n} + 2\pi \alpha \sigma \sqrt{\frac{2\log(2\zeta^{-1}) ( \theta + R^2)}{n}}.
	\end{equation} 
\end{prop}

The quantities $\theta$ and $\Gamma$ help characterize the geometry of the constraint set in \eqref{eq:gen_mat_lasso} for a given observation model. In particular, \eqref{eq:qub} describes how well the $\ell_2$-norm of the measurements are concentrated around that of the ground truth, and \eqref{eq:gaussian_complexity} describes the Gaussian complexity of the measurement operator over the convex set of candidate matrices. In order to obtain tight estimates of $\theta$ and $\Gamma$ for specific applications, we first interpret the matrix as an operator between suitably chosen Banach spaces. We then obtain tight upper bounds on $\theta$ and $\Gamma$ by computing the entropy numbers of operators between particular tensor products of the chosen Banach spaces. In the following sections, we provide upper bounds on the estimation error for the two applications, decentralized sketching and low-rank matrix completion. These error bounds are particular realizations of Proposition \ref{prop:deterministic_errbnd}.

\subsection{Statistical analysis of decentralized sketching}
\label{subsec:main}

Our main result provides an error bound for the estimator by \eqref{eq:gen_mat_lasso} when it applies to the decentralized sketching problem. 

\begin{theorem}
\label{thm:eb_ds}
Let $A_k = \sqrt{d_2} b_k e_{j_k}^\top$ with $j_k \in [d_2]$ satisfying $j_k \equiv k$ modulo $[d_2]$ for $k=1,\dots,Ld_2$ and $\tnorm{\cdot}$ be defined as in \eqref{eq:tnorm_mixed}. 
Suppose that $b_1,\dots,b_{Ld_2}$ are i.i.d. $\mathcal{N}(0,I_{d_1})$ and $\eta_1,\dots,\eta_{Ld_2}$ are i.i.d. $\mathcal{N}(0,\sigma^2)$. 
Then the following statement holds with probability $1 - (d_1+d_2)^{-1}$: 
For any matrix $M_0 \in K_\mathrm{mixed} := \{ M \in \mathbb{R}^{d_1 \times d_2} : \norm{M_0}_{1\to 2} \leq \alpha, \norm{M_0}_\mathrm{mixed} \leq \sqrt{r} \alpha\}$, the estimate $\Phi$ by \eqref{eq:gen_mat_lasso} satisfies 
\begin{equation}
\label{eq:ub1:ds}
\frac{1}{d_1 d_2} \|\Phi - M_0\|_\mathrm{F}^2 
\lesssim \frac{\alpha^2}{d_1} \cdot \max\left(1, \frac{\sigma\sqrt{\log(Ld_2)}}{\alpha\sqrt{d_2}}\right) \cdot \sqrt{\frac{r(d_1+d_2)\log^4(d_1+d_2)}{Ld_2}}. \end{equation}
\end{theorem}

Let $\alpha = \mu/\sqrt{d_2}$ for some $\mu \geq 1$. 
Then the intersection of $K_\mathrm{mixed}$ with the unit Frobenius-norm sphere $\mathbb{S}_\mathrm{F}$ coincides with the intersection of $\mathcal{C}_\mathrm{mixed} := \{M : \tnorm{M} \leq (\mu/\sqrt{d_2}) \norm{M}_\mathrm{F}\}$ with $\mathbb{S}_\mathrm{F}$. 
The latter set consists of unit-Frobenius norm matrices which are approximately rank-$r$ and the maximum of column norms is not too larger compared to the root-mean-square. 
The error bound in \eqref{eq:ub1:ds} applies to $\mathcal{C}_\mathrm{mixed} \cap \mathbb{S}_\mathrm{F}$. 
It is straightforward to verify that
\[
\mu = \sup_{M \in \mathcal{C}_\mathrm{mixed}\setminus\{0\}} \frac{\sqrt{d_2} \norm{M}_{1\to 2}}{\norm{M}_\mathrm{F}}.
\]
Then $\mu$ denotes the ``spikiness'' of the column norms. 
With 
\[
\text{SNR} = \frac{\sum_{k=1}^{Ld_2} \mathbb{E}[\mathrm{tr}(A_k^\top M_0)^2]}{\sum_{k=1}^{Ld_2} \mathbb{E}[\eta_k^2]}
= \frac{\norm{M_0}_\mathrm{F}^2}{\sigma^2},
\]
the error bound in \eqref{eq:ub1:ds} is rewritten as 
\begin{equation}
\label{eq:ub2:ds}
\frac{\|\Phi - M_0\|_\mathrm{F}^2}{\norm{M_0}_\text{F}^2}
\lesssim \mu^2 \cdot \max\left(1, \frac{\mu^{-1} \log(Ld_2)}{\text{SNR}^{1/2}} \right) \cdot \sqrt{\frac{r(d_1+d_2)\log^4(d_1+d_2)}{Ld_2}}.
\end{equation}
Particularly, in a high noise regime where $\text{SNR} = O(\mu^{-2} \log(Ld_2))$, we note that \eqref{eq:ub2:ds} further reduces to 
\begin{equation}
\label{eq:ub3:ds}
\frac{\|\Phi - M_0\|_\mathrm{F}^2}{\norm{M_0}_\text{F}^2}
\lesssim \sqrt{\frac{\text{SNR}^{-1} \cdot \mu^2 r(d_1+d_2)\log^4(d_1+d_2)}{Ld_2}}. 
\end{equation}
Since the error bound in \eqref{eq:ub3:ds} is invariant under scaling of $M_0$, it indeed applies to $\mathcal{C}_\mathrm{mixed}$, which consists of all non-spiky and approximately rank-$r$ matrices. 
That is, in order to achieve an $\epsilon$-accurate estimation in the normalized error, it suffices to obtain $\widetilde{O}(\epsilon^{-2} \cdot \text{SNR}^{-1} \cdot \mu^2 r(d_1+d_2))$ random local measurements. 
To show the tightness of the error bound in Theorem~\ref{thm:eb_ds}, we compare it to a matching lower bound given in the following theorem. 

\begin{theorem}
\label{thm:minimax_ds}
For the measurement model in \eqref{eq:meas_mdl}, suppose that the parameters satisfy 
\begin{equation*}
   \frac{48 \alpha^2}{d_1 \vee d_2} \leq \alpha^2 r \leq \frac{\sigma^2 d_1}{128 L}.
\end{equation*} 
Then the minimax $\norm{\cdot}_\mathrm{F}$-risk is lower-bounded as
\begin{equation}
\label{eq:minimax_decentralized}
    {\inf_{\Phi}} \sup_{M_0 \in K_\mathrm{mixed}} \frac{\E \|\Phi - M\|_\mathrm{F}^2 }{d_1 d_2} \geq \frac{\alpha^2}{16 d_1} \left( 1 \wedge \frac{\sigma}{\alpha\sqrt{d_2}} \sqrt{\frac{r(d_1 + d_2)}{L d_2}} \right).
\end{equation}
Further, when $Ld_2 > r(d_1+d_2)$, we have that 
\begin{equation}
    \label{eq:minimax_reformatted}
     {\inf_{\Phi}} \sup_{M_0 \in K_\mathrm{mixed}} \frac{\E \|\Phi - M\|_\mathrm{F}^2 }{d_1 d_2} \geq \frac{\alpha^2}{16 d_1} \sqrt{\frac{r(d_1+d_2)}{Ld_2}} \left( 1 \wedge \frac{\sigma}{\alpha\sqrt{d_2}} \right).
\end{equation}
\end{theorem}

Note that, in a low SNR regime, the error bound in \eqref{eq:ub1:ds} matches the minimax lower bound in \eqref{eq:minimax_decentralized} up to a logarithmic factor. 
However, in a high SNR regime, where the $\mathrm{SNR}^{1/2} \mu^{-1}$ term is $o(1)$, the error bound in \eqref{eq:ub1:ds} is suboptimal. 
In particular, the estimator does not provide exact parameter recovery in the noiseless case ($\sigma = 0$). 

The only results in the literature that are directly comparable to Theorem~\ref{thm:minimax_ds} are \citep{nayer2019phaseless,nayer2021fast}.  These works show how $M_0$ can be recovered using an alternating minimization or gradient descent algorithm.  When the unknown matrix is exactly rank-$r$ and there is no noise in the measurements, they showed that alternating minimization (resp. gradient descent) provides an $\epsilon$-accurate estimate from $\widetilde{O}(\tilde{\mu}^2\max(d_1r^3,d_2r)\log(1/\epsilon))$ (resp. $\widetilde{O}(\tilde{\mu}^2 r^2(d_1+d_2) \log(1/\epsilon))$) measurements, where the incoherence parameter $\tilde{\mu} = \|V^\top\|_{1\to 2} \sqrt{d_2/r}$ corresponds to the ``spikiness'' of the right singular vectors. 
We compare these results to Theorem~\ref{thm:eb_ds} in the following perspectives. 

\textbf{Low-spikiness vs incoherence:} 
The incoherence parameter $\tilde{\mu}$ on singular vectors is upper-bounded by the product of the condition number $\kappa$ and the spikiness $\mu$ of $M_0$. 
This implies that their sample complexity depends on conditioning of the unknown matrix $M_0$. 
This dependence has been pointed out as a weakness of nonconvex algorithms for low-rank recovery compared to the convex counterpart. 
It might be possible to alleviate this dependence by adopting a scaled gradient descent algorithm as in \citep{tong2020accelerating}. 
However, in the presence of noise, the spectral initialization, on which the theoretical performance guarantees of these algorithms depend critically, cannot be made free from the dependence on $\kappa$. 
Furthermore, it is easy to construct matrices such that $\tilde{\mu}$ is much larger than $\mu$.  To see this, let $M$ be of rank $r-1$. 
Suppose that all left (resp. right) singular vectors are orthogonal to $u \in \mathbb{S}^{d_1-1}$ (resp. $e_1 \in \mathbb{S}^{d_2-1}$). 
Then $M_0 = M + t u e_1^\top$ is of rank-$r$. 
It follows that the right singular vectors of $M_0$ are given as those of $M$ and $e_1$. 
For small $t$, the spikiness $\mu$ of $M_0$ is similar to that of $M$. 
However, the incoherence $\tilde{\mu}$ of $M_0$ is close to the maximal value $\sqrt{d_2/r}$. 
In this scenario, if $\mu = O(1)$ and $r = O(1)$, then the sample complexity for alternating minimization is larger than that for the convex estimator in \eqref{eq:gen_mat_lasso} in order. 
Also note that the condition number $\kappa$ is very large for small $t$. 
Therefore, one may deduce that the convex estimator has advantages in its stable operation regardless of conditioning of $M_0$.

\textbf{Flexibility of model:} 
In an ideal scenario where $M_0$ is exactly rank-$r$ and measurements are noise-free, if one ignores the difference between $\mu$ and $\tilde{\mu}$, then the sample complexity for alternating minimization and gradient descent \citep{nayer2019phaseless,nayer2021fast} is significantly better than that in Theorem~\ref{thm:eb_ds} in terms of their dependence on $\epsilon$. 
To get within accuracy $\epsilon$, \citep{nayer2019phaseless,nayer2021fast} require $O(\log(1/\epsilon))$ measurements whereas Theorem \ref{thm:eb_ds} requires $O(1/\epsilon^2)$. However, Theorem \ref{thm:eb_ds} provides a uniform error bound over \textit{all} matrices that are \textit{approximately} low-rank. It is unclear how the results in \citep{nayer2019phaseless,nayer2021fast} could be extended to this more flexible model.

Furthermore, our result shows that the convex estimator in \eqref{eq:gen_mat_lasso} is consistent and provides a near optimal error bound in the presence of (not too weak) noise.
This result matches the minimax lower bound presented in Theorem \ref{thm:minimax_ds}.

\subsection{Low-rank matrix completion without incoherence conditions}

As shown earlier, a local measurement operator also arises in matrix completion. 
Even in the absence of noise, the unique identification of $M_0$ in matrix completion has been shown only under certain incoherence conditions imposed on its singular vectors. 
However, in practice, it is often not possible to verify that such incoherence conditions are satisfied by the ground-truth matrix $M_0$. 
On the contrary, in many applications like collaborative filtering, the entries of the unknown matrix are bounded by a certain threshold. 
Motivated by this observation, \citet{negahban2012restricted} considered an alternative estimation problem, in which the unknown matrix satisfies the above milder condition.
Specifically, they demonstrated that a LASSO estimator with the nuclear norm provides the minimax-optimal error bound when the noise level is higher than some threshold. 
Later the max-norm has been proposed as an alternative regularizer to the nuclear norm \citep{srebro2004maximum,srebro2005rank,linial2007complexity}. 
Particularly for completing low-rank matrices with bounded entries, it has been shown that the max-norm regularized estimator empirically outperforms that with the nuclear norm \citep{srebro2004maximum}.
\citet{foygel2011concentration} presented a near-optimal error bound. 
This result was extended to sampling with a non-uniform distribution and sharpened to an optimal error bound by \citet{cai2016matrix}. 

We provide an alternative derivation of the statistical analysis of the max-norm LASSO estimator for matrix completion. 
\citet{cai2016matrix} employed Bousquet’s version of Talagrand’s concentration inequality for empirical processes indexed by bounded functions. 
Indeed, this enabled to drop the logarithmic factor in the previous result \citep{foygel2011concentration}.
Our approach is based on a unifying characterization of the $\gamma_2$ norm via the projective tensor norm and Maurey's empirical method by \citet{carl1985inequalities}. 
Although our error bound resulted in an extra poly-log factor, we believe that this alternative analysis can be useful to understand other regression problems with a local measurement operator sharing similar structure. 
The result is presented in Section~\ref{sec:err_bnd}.

\subsection{Summary of contributions and related work}

In this paper, we present a unifying design principle and statistical analysis for a convex estimator for approximately low-rank matrices from local noisy measurements. 
The results are given as a near-optimal error bound specifically for two illustrating examples of decentralized sketching and matrix completion. 
The convex regularizer we design adapting to the structure in local measurement operators generalizes the max-norm and mixed-norm to a set of tensor norms. 
It has been empirically shown that tensor norms beyond the two examples also provide successful regularization for denoising \citep{bruer2017recovering}.

The relations among various tensor norms of rank-constrained linear operators are derived for selected pairs of Banach spaces via fundamental properties studied in classical functional analysis \citep{jameson1987summing}. 
In particular, we show that the mixed-norm and regularized inference with respect to it can be rewritten as a standard semidefinite program. 

Furthermore we present upper bounds on the entropy integral with respect to covering number between tensor products of Banach spaces via Maurey's empirical method \citep{carl1985inequalities}. 
This result leads to a version of restricted isometry property for the local linear operators arising in decentralized sketching and matrix completion. 
Finally, we show that these concentration results provide a near optimal error bound for the corresponding applications. 

Besides the aforementioned related work on decentralized sketching and matrix completion, there has been discussions on related problems in the literature. 
Multilinear regression \citep{negahban2011estimation} and sketching low-rank covariance \citep{anaraki2014memory,azizyan2018extreme} share similar structure to decentralized sketching. 
However, the problem settings and objectives are different from decentralized sketching in various perspectives and the results are not directly comparable. 
\citet{kliesch2016improving} considered a set of novel ideas on a related regularization problem. 
They extended the nuclear norm to square and diamond norms. The former improved the stability of inverse problem by optimizing the descent cone of the regularizer and the latter enabled regularization applies to a more general objects in operator spaces, which was inspired from key applications in quantum tomography. 

Another related work \citep{bruer2017recovering} studied the use of various tensor norms as a regularizer for denoising low-rank matrices. 
He considered various products of norms on the left and right factors of a low-rank matrix, which are chosen according to the underlying structure of the factors. 
Nonconvex optimization algorithms were proposed for denoising with these regularizers. It has been shown that a suitably chosen tensor norm empirically outperforms the other norms. 
This study is different from our approach since our choice of tensor norms is based on the inherent structure in the measurement process.

The rest of this paper is organized as follows: Preliminaries on tensor products and tensor norms are provided in Section~\ref{sec:prelim}.
The choice of tensor norms for matrix completion and descentralized sketching is discussed in the language of tensor products in Section~\ref{sec:estimators}.
This is followed by the analysis of tensor norms on selected pairs of Banach spaces in Section~\ref{sec:rankineq}. 
We present the covering number analysis of tensor products in Section~\ref{sec:ent_est}. Then we apply these results to get an error bound in Section~\ref{sec:err_bnd}. 
A matching information-theoretic lower bound is derived in Section~\ref{sec:minimax}. 
We conclude the paper with discussions and future directions.

\section{Notations and preliminaries}
\label{sec:prelim}

In this section, we introduce notations used throughout the paper, and recall the definition of a tensor product of two Banach spaces and various norms defined on it.

For a positive integer $d$, let $[d]$ denote the set $\{1,2,\dots,d\}$. 
For real numbers $a$ and $b$, let $a \vee b$ and $a \wedge b$ denote the maximum and minimum of $\{a,b\}$. 
For a matrix $M$, its transposition is denoted by $M^\top$. 
For a linear operator $\mathcal{A}$, we use $\mathcal{A}^*$ to denote the adjoint operator of $\mathcal{A}$. 
For a vector space $X$, the dual space consisting of all linear functionals on $X$ is denoted by $X^*$.
We use $\otimes$ for both tensor product and Kronecker product, the distinction will be clear in the context. 
By $a \lesssim b$, we mean that there is an absolute constant $C$ such that $a \leq C b$. 
Throughout the paper, $c, C, c_k, C_k$ will denote absolute constants which might vary line to line. 

\subsection{Algebraic tensor product}

The algebraic tensor product of two vector spaces $X$ and $Y$, denoted by $X \otimes Y$, is a set of all finite sum of outer products of two vectors respectively from $X$ and $Y$. 
An element in $X \otimes Y$ acts as a bilinear function on $X^* \times Y^*$, i.e. 
\[
\left(\sum_{k=1}^r x_k \otimes y_k \right) (x^* \otimes y^*) = \sum_{k=1}^r x^*(x_k) y^*(y_k),
\]
where $x_1, \dots, x_r \in X$, $y_1, \dots, y_r \in Y$, $x^* \in X^*$, and $y^* \in Y^*$. 
Alternatively, an element in $X \otimes Y$ acts as a linear operator mapping $X^*$ to $Y$, i.e. 
\[
\left(\sum_{k=1}^r x_k \otimes y_k \right) x^* = \sum_{k=1}^r x^*(x_k) y_k. 
\]
Indeed, $X \otimes Y$ is embedded into the set $B(X^*,Y^*)$ of all bilinear functions on $X^* \times Y^*$ or into the set $L(X^*,Y)$ of all linear operators from $X^*$ to $Y$ via an isomorphism. 
Particularly when $X$ and $Y$ are finite-dimensional, then the aforementioned inclusions become equality. 
More details are referred to the monographs by \citet{defant1992tensor} and by 
\citet{diestel2008metric}. 

In this paper, we only consider the finite-dimensional case. 
Thus a matrix $M \in \mathbb{R}^{m \times n}$, as the matrix representation of a linear operator from a vector space of dimension $n$ to another vector space of dimension $m$, is interpreted as an element in the corresponding tensor product of vector spaces.

\subsection{Tensor norms}

For Banach spaces $X$ and $Y$, there exist various ways to define a norm on $X \otimes Y$. 
A norm $\norm{\cdot}$ on $X \otimes Y$ is called a \textit{tensor norm} \citep[p. 147]{defant1992tensor} or \textit{reasonable crossnorm} \citep{grothendieck1956resume} \cite[p. 5]{diestel2008metric} if it satisfies 
\[
\norm{x \otimes y} \leq \norm{x}_X \norm{y}_Y, \quad \forall x \in X, \, y \in Y,
\]
and its dual norm, denoted by $\norm{\cdot}_*$, satisfies
\[
\norm{x^* \otimes y^*}_* \leq \norm{x^*}_{X^*} \norm{y^*}_{Y^*}, \quad \forall x^* \in X^*, \, y^* \in Y^*.
\]
Throughout this paper, we utilize a selected set of tensor norms summarized below. \\

\noindent\textbf{Injective norm: } The \textit{injective norm} is induced by viewing $X \otimes Y$ as a subspace of $B(X^*,Y^*)$, i.e.
\[
\|T\|_\vee := 
\sup_{x^* \in B_{X^*}, y^* \in B_{Y^*}} |\langle x^* \otimes y^*, T\rangle|,
\]
where $B_{X^*}$ and $B_{Y^*}$ are the unit balls in $X^*$ and $Y^*$ respectively \citep[p. 46]{defant1992tensor}. 
Then $X \mathbin{\widecheck{\otimes}} Y$ will denote the corresponding Banach space obtained via completion of $X \otimes Y$ with respect to the injective norm. 
The injectivity of the injective norm refers to the following property: If $Z$ is a closed subspace of $X$, then $Z \mathbin{\widecheck{\otimes}} Y$ is also a closed subspace of $X \mathbin{\widecheck{\otimes}} Y$ \citep[Proposition~1.1.6]{diestel2008metric}. 
Furthermore, the injective tensor norm is the smallest tensor norm on $X \otimes Y$.
Since $X \otimes Y$ is also isometrically isomorphic to a subspace of $L(X^*,Y)$, the injective tensor norm of $T$ coincides with the operator norm of $T \in L(X^*,Y)$. \\

\noindent\textbf{Projective norm: } The \textit{projective norm} is the largest tensor norm on $X \otimes Y$ and is defined as the Minkowski gauge function of the absolutely convex hull of $B_X \otimes B_Y := \{x \otimes y ~:~ x \in B_X, ~ y \in B_Y\}$, where $B_X$ and $B_Y$ denote the unit ball in $X$ and $Y$, respectively \cite[p. 27]{defant1992tensor}. 
Then the completion of $X \otimes Y$ with respect to the projective norm gives the Banach space $X \mathbin{\widehat{\otimes}} Y$. 
The projectivity of the projective norm implies that for a subspace $Z$ of $X$, $(X/Z) \mathbin{\widehat{\otimes}} Y$ is a quotient of $X \mathbin{\widehat{\otimes}} Y$, where $X/Z$ denotes the quotient of $X$ with respect to $Z$ \citep[Proposition~1.1.7]{diestel2008metric}. 
It has been shown (e.g. \cite{grothendieck1956resume}, \citep[Proposition 1.1.4]{diestel2008metric}, \citep[p. 27]{defant1992tensor}) that the projective norm can be computed as 
\[
\|T\|_\wedge = \inf \left\{ \sum_{k=1}^n \|x_k\|_X \|y_k\|_Y \, : \, n \in \mathbb{N}, \, T = \sum_{k=1}^n x_k \otimes y_k \right\}.
\]
It has been shown \citep[p.~3]{pisier1986factorization} \citep[p.~19]{jameson1987summing} that if $X$ or $Y$ is finite-dimensional, then the projective norm coincides with the \textit{$1$-nuclear norm} defined by
\[
\nu_1(T) := \inf \left\{ \sum_{k=1}^\infty \|x_k\|_X \|y_k\|_Y \, : \, T = \sum_{k=1}^\infty x_k \otimes y_k \right\}.
\]

The injective and projective tensor products provide a duality as follows: $X \mathbin{\widecheck{\otimes}} Y$ is isometrically embedded into a subspace of $(X^* \mathbin{\widehat{\otimes}} Y^*)^*$ \cite[p. 47]{defant1992tensor}. Furthermore, if $X$ and $Y$ are finite-dimensional, then $X \mathbin{\widecheck{\otimes}} Y= (X^* \mathbin{\widehat{\otimes}} Y^*)^*$. 

There exist tensor norms induced by viewing $T \in X \otimes Y$ as a linear operator in $L(X^*,Y)$. 
We consider the following Banach-space-valued sequence spaces. 
Let $\ell_2^\mathrm{strong}(X)$ denote the Banach space of strongly $2$-summable $X$-valued sequences equipped with the norm $\norm{(X_k)}_{\ell_2^\mathrm{strong}(X)} = \norm{(\norm{x_k}_X)}_2$. 
Let $\ell_2^\mathrm{weak}(X)$ denote the Banach space of weakly $2$-summable $X$-valued sequences equipped with the norm $\norm{(x_k)}_{\ell_2^\mathrm{weak}(X)} = \sup_{x^* \in B_{X^*}} (\sum_k |x^*(x_k)|^2)^{1/2}$. 
Let $\tilde{T}$ be the linear operator acting on an $X^*$-valued sequence $(x^*_k)$ via coordinate-wise application of $T$. \\

\noindent\textbf{$2$-summing norm: } 
Then the \textit{$2$-summing norm} of $T$, denoted by $\pi_2(T)$, is defined as the operator norm of $\tilde{T}: \ell_2^\mathrm{weak}(X^*) \rightarrow \ell_2^\mathrm{strong}(Y)$ \citep{diestel1985introduction}. 
In other words, $\pi_2(T)$ is computed as the smallest constant $c > 0$ that satisfies
\[
\sum_k \|T x^*_k\|_Y^2
\leq c^2 \, \sup_{x \in B_{X^{**}}} \sum_k |\langle x, x^*_k\rangle|^2
\]
for all sequences $(x^*_k) \subset X^*$. 
In particular, if $X$ and $Y$ are Hilbert spaces, $\pi_2(T)$ coincides with the Hilbert-Schmidt norm, which is the Frobenius norm when $T$ is a finite matrix \citep[Proposition~1.9]{pisier1986factorization}. \\

\noindent\textbf{$\gamma_2$ norm: } 
Again, viewing $X \otimes Y$ as a subspace of $L(X^*,Y)$, the \textit{$\gamma_2$ norm} of $T \in X \otimes Y$ is induced by factorization via an Hilbert space as follows \citep[p.~21]{pisier1986factorization}:
\begin{align*}
\gamma_2(T) &:= \inf \{ \|T_1\| \|T_2\| \,:\, d \in \mathbb{N}, T_1 \in B(X^*,\ell_2^d), \, T_2 \in B(\ell_2^d,Y), \, T = T_2 T_1 \}.
\end{align*}
Alternatively, $\gamma_2(T)$ is also written as
\[
\gamma_2(T) = \inf \left\{ 
\sup_{x^* \in B_{X^*}} \left( \sum_k |x^*(x_k)|^2 \right)^{1/2} 
\sup_{y^* \in B_{Y^*}} \left( \sum_k |y^*(y_k)|^2 \right)^{1/2} \,:\, T = \sum_k x_k \otimes y_k
\right\}.
\]
Due to \citep[Theorem~2.4]{pisier1986factorization}, $\gamma_2(T)$ can be computed as the smallest constant $c > 0$ that satisfies
\[
\sum_k \|T z^*_k\|_Y^2
\leq c^2 \, \sum_k \norm{x^*_k}_{X^*}^2
\]
for all sequences $(x^*_k), (z^*_k) \subset X^*$ such that $(x^*_k)$ dominates $(z^*_k)$ in $\ell_2^\mathrm{weak}(X^*)$, i.e. 
\[
\sup_{x \in B_{X^{**}}} \sum_k |\langle x, z^*_k\rangle|^2
\leq
\sup_{x \in B_{X^{**}}} \sum_k |\langle x, x^*_k\rangle|^2.
\] 
Therefore, it always hold that $\gamma_2(T) \leq \pi_2(T)$. \\

\noindent\textbf{Relation to matrix norms: } The above tensor norms reduce to well-known matrix norms in some special cases. 
Let us consider the first example where $X = \ell_2^n$ and $Y = \ell_2^m$. 
Then $\norm{T}_\vee$ and $\gamma_2(T)$ become the spectral norm of $T \in \mathbb{R}^{m \times n}$, which is the largest singular value. 
Additionally, $\norm{T}_\wedge$ becomes the nuclear norm of $T$, i.e. the sum of all singular values. Moreover, $\pi_2(T)$ becomes the Frobenius norm. 
Next, we consider the case where $X = \ell_\infty^n$ and $Y = \ell_\infty^n$. 
First $\norm{T}_\vee$ becomes the largest entry of $T \in \mathbb{R}^{n \times n}$ in magnitude. 
Let $(a_{ij})$ denote the $(i,j)$th entry of $T \in \mathbb{R}^{n \times n}$, i.e. $T = \sum_{i,j=1}^n a_{ij} e_i e_j^\top$, where $(e_i)$ corresponds to the canonical basis for $\mathbb{R}^n$. 
Furthermore, the projective norm is computed by
\begin{equation}
\label{eq:projnorm_infinf}
\norm{T}_\wedge = \sup \left\{ \left| \sum_{i,j=1}^n a_{ij} \alpha_i \beta_j \right| \,:\, \sup_{i \in [n]} |\alpha_i| \leq 1, \, \sup_{i \in [n]} |\beta_i| \leq 1 \right\}.
\end{equation}
Finally, $\gamma_2(T)$ is represented as 
\begin{equation}
\label{eq:gamma2_infinf_littlgGT}
\gamma_2(T) = \inf_{H} \left\{ \sup_{i \in [n]} \norm{u_i}_H \sup_{j \in [n]} \norm{v_j}_H \,:\, (u_i), (v_j) \subset H, \, a_{ij} = \langle u_i, v_j \rangle, \forall i,j \in [n] \right\},
\end{equation}
where the infimum is taken over all Hilbert spaces.

\section{Tensor-norm-based estimators}
\label{sec:estimators}

In this section, we justify the choice of matrix norms for matrix completion and decentralized sketching in the language of tensor product.
We first recall the known results on matrix completion via $\ell_\infty^n \otimes \ell_\infty^m$ in the literature. \citet{linial2007complexity} showed that the $\gamma_2$ norm on $\ell_\infty^n \otimes \ell_\infty^m$ is upper-bounded by the operator norm multiplied by the square root of the rank. In fact, their result is derived from the fact that the Banach-Mazur distance (see e.g. \citep{pisier1999volume})  between a finite-dimensional Banach space and a Hilbert space is no larger than the square root of the dimension. 
Therefore their result \cite[Lemma~4.2]{linial2007complexity} applies to any pair of Banach spaces. 
We paraphrase it as the following lemma. 

\begin{lemma}[{A paraphrased version of \cite[Lemma~4.2]{linial2007complexity}}]
\label{lemma:linial}
Suppose that $T \in X \otimes Y$ with Banach spaces $X,Y$ satisfies $\mathrm{rank}(T) \leq r$. Then
\[
\|T\| \leq \gamma_2(T) \leq \sqrt{r} \|T\|.
\]
\end{lemma}
In the case of $\ell_\infty^n \otimes \ell_\infty^m$, the $\gamma_2$ norm is equal to the max norm defined in \eqref{eq:def_max_norm}.
Further, it has been shown \cite{srebro2004maximum} that the max-norm is computable via a semi-definite formulation given by
\begin{equation}
\label{eq:max_norm_semidefinite}
\begin{array}{lll}
\|M\|_{\max}
= & \displaystyle \underset{W_1, W_2}{\inf} & \max \left ( \| \mathrm{diag}(W_1)\|_{\infty} , \|\mathrm{diag}(W_2) \|_\infty \right ) \\
& \mathrm{s.t} & \begin{bmatrix} W_1 & M\\ M^\top & W_2 \end{bmatrix} \succeq 0.
\end{array}
\end{equation} 
Therefore, the estimator in \eqref{eq:gen_mat_lasso} for matrix completion with the regularizer in \eqref{eq:tnorm_max} can be rewritten as a standard form of a semidefinite program (SDP) given by
\begin{equation}
\label{eq:max_norm_estimator_SDP}
\begin{array}{ll}
\displaystyle \mathop{\text{minimize}}_{W_1, W_2, M} & \|y - \mathcal{A}(M)\|_2^2 \\
\text{subject to} & \|M\|_\infty \leq \alpha, \ \|\mathrm{diag}(W_1)\| \leq R, \ \|\mathrm{diag}(W_2)\| \leq R,
\\[5pt]
& \begin{bmatrix} W_1 & M\\ M^\top & W_2 \end{bmatrix} \succeq 0.
\end{array}
\end{equation}
\citet{cai2016matrix} proposed an ADMM (alternating direction method of multipliers) algorithm based on an equivalent convex formulation, which scales more efficiently than general purpose solvers for SDP. 
In the remainder of this section, we will establish analogous results for decentralized sketching via the tensor product $\ell_\infty^n \otimes \ell_2^m$. 
Although Lemma~\ref{lemma:linial} also applies to $\ell_\infty^n \otimes \ell_2^m$, it has not been known whether the $\gamma_2$ norm on $\ell_\infty^n \otimes \ell_2^m$ can be computed easily as a standard convex optimization problem. 
Therefore, we consider a different tensor norm on $\ell_\infty^n \otimes \ell_2^m$, which is polynomially computable and provides a proxy of rank similar to Lemma~\ref{lemma:linial} simultaneously.

\subsection{Low-rankness inducing property of mixed norm}

First, we show that analogous relationships exist between another pair of tensor norms for rank-$r$ linear operators in $\ell_\infty^n \otimes \ell_2^m$. 
Indeed, we start with a slightly generalized result that applies to $X \otimes \ell_2^m$ for a general Banach space $X$.
The following lemma, proved in Appendix~\ref{sec:proof:lemma:pi2a_op}, shows that the 2-summing norm of the adjoint does not exceed the operator norm multiplied by the square root of the rank.

\begin{lemma}
\label{lemma:pi2a_op}
Let $T \in X \otimes \ell_2^m$, where $X$ is a Banach space. 
Suppose that $T^*$ is 2-summing and $\mathrm{rank}(T) \leq r$. Then we have
\[
\|T\| \leq \pi_2(T^*) \leq \sqrt{r} \|T\|.
\]
\end{lemma}
\noindent The result in Lemma~\ref{lemma:pi2a_op} applies beyond the case $X = \ell_\infty^n$. 
For example, when $X = \ell_2^n$, the 2-summing and operator norms become the Frobenius and spectral norms, respectively. 

Next we show that $\pi_2(T^*)$ for $T \in \ell_\infty^n \otimes \ell_2^m$ is written through an optimal factorization via a Hilbert space similarly to the $\gamma_2$ norm on the tensor product $\ell_\infty^n \otimes \ell_\infty^m$. 
The following lemma is used to derive the desired result. Its proof is delegated to Appendix~\ref{sec:proof:lemma:p2ANDp}. 
\begin{lemma}
\label{lemma:p2ANDp}
Let $T \in X \otimes Y$ with $X$ complete. Then
\begin{equation}
\label{eq:pi2Tstar1}
\begin{aligned}
\pi_2(T^*)
& := \inf \{ \pi_2(T_1^*) \|T_2^*\| \,:\, d \in \mathbb{N}, \, T_1^* \in L(Y^*,\ell_2^d), \, T_2^* \in L(\ell_2^d,X), \, T^* = T_2^* T_1^* \}.
\end{aligned}
\end{equation}
\end{lemma}
\noindent In a special case when $T \in \ell_\infty^n \otimes \ell_2^m$, the 2-summing norm of $T^*$ is computed by the following optimization problem:
\begin{equation}
\label{eq:pi2Tstar2}
\begin{aligned}
\pi_2(T^*)
& = \inf \{ \|T_1\|_\mathrm{F} \|T_2\| \,:\, d \in \mathbb{N}, \, T_1 \in L(\ell_2^d,\ell_2^m), \, T_2 \in L(\ell_1^n,\ell_2^d), \, T = T_1 T_2 \}.
\end{aligned}
\end{equation}
The expression in \eqref{eq:pi2Tstar2} is derived as follows: 
It trivially holds that $\|T_2^*\| = \|T_2\|$. 
Furthermore, for every $T_1^* \in L(\ell_2^m,\ell_2^d)$, we have
\[
\pi_2(T_1^*) = \|T_1^*\|_\mathrm{F} = \|T_1\|_\mathrm{F} = \pi_2(T_1).
\]
Plugging in these results to \eqref{eq:pi2Tstar1} provides \eqref{eq:pi2Tstar2}. 

Note that the $2$-summing norm of $T^*$ in \eqref{eq:pi2Tstar2} coincides with the definition of the mixed norm in \eqref{eq:def_mixed_norm}. 
Therefore, due to Lemma~\ref{lemma:pi2a_op}, we have
\[
\norm{M}_{1\to 2} \leq \norm{M}_\mathrm{mixed} \leq \sqrt{r} \norm{M}_{1\to 2} 
\]
for any matrix of rank $r$. 
In other words, the mixed norm plays a proxy on the rank when the input matrix is normalized by the max-column norm. 
This is analogous to the fact that $\norm{M}_\mathrm{F} \leq \norm{M}_* \leq \sqrt{r} \norm{M}_\mathrm{F}$ holds for all rank-$r$ matrices and justifies why the regularizer with respect to the matrix norm in \eqref{eq:tnorm_mixed} induces a low-rank solution to \eqref{eq:gen_mat_lasso}.

\subsection{Semidefinite program characterization of mixed norm}

Next we show that the mixed norm in \eqref{eq:pi2Tstar2} can be expressed as a semidefinite program. 
Note that if $T_1 T_2 = T$, then $T_1' = a T_1$ and $T_2' = a^{-1} T_2$ for any $a \neq 0$ also satisfy $T_1' T_2' = T$. 
Therefore, from the arithmetic-geometric-harmonic-means inequality, it follows that \eqref{eq:pi2Tstar2} is equivalently rewritten as 
\begin{equation*}
\begin{aligned}
\pi_2(T^*)
& = \frac{1}{2} \inf \{ \|T_1\|_\mathrm{F}^2 + \|T_2\|^2 \,:\, d \in \mathbb{N}, \, T_1 \in L(\ell_2^d,\ell_2^m), \, T_2 \in L(\ell_1^n,\ell_2^d), \, T = T_1 T_2 \}.
\end{aligned}
\end{equation*}
We also have
\begin{align*}
    \|T_1\|_\mathrm{F}^2 = \mathrm{trace}(T_1T_1^*) = \langle T_1, T_1 \rangle 
    \quad \text{and} \quad 
    \|T_2 \|^2 = \|\mathrm{diag}(T_2^* T_2) \|_\infty,
\end{align*} 
where $\mathrm{diag}(\cdot)$ constructs a column vector consisting of the diagonal entries of the input matrix.
Furthermore, by the proof of \cite[Lemma~3]{srebro2004maximum}, there exist $T_1$ and $T_2$ such that $T = T_1 T_2^*$, $W_1 = T_1 T_1^*$, and $W_2 = T_2 T_2^*$ if and only if 
\[
\begin{bmatrix} W_1 & T \\ T^* & W_2 \end{bmatrix} \succeq 0.
\]
Combining these results implies that $\pi_2(T^*)$ is written as a semidefinite program given by
\begin{equation}
\label{eq:mix_norm_semidefinite}
\begin{array}{lll}
\pi_2(T^*)
= & \displaystyle \underset{W_1, W_2}{\inf} & \max \left ( \mathrm{trace} (W_1) , \|\mathrm{diag}(W_2) \|_\infty \right ) \\
& \mathrm{s.t} & \begin{bmatrix} W_1 & T\\ T^* & W_2 \end{bmatrix} \succeq 0.
\end{array}
\end{equation}

\subsection{Algorithms for convex estimator} 

The semidefinite characterization presented in \eqref{eq:mix_norm_semidefinite} leads to a practical algorithm for the estimator in \eqref{eq:gen_mat_lasso} for decentralized sketching with respect to the norm in \eqref{eq:tnorm_mixed} as 
\begin{equation}
\label{eq:mix_norm_estimator_SDP}
\begin{array}{ll}
\displaystyle \mathop{\text{minimize}}_{W_1,W_2,M} & \|y - \mathcal{A}(M)\|_2^2 \\
\text{subject to} & \|M\|_{1 \rightarrow 2} \leq \alpha, \ \mathrm{trace}(W_1) \leq R, \ \|\mathrm{diag}(W_2)\| \leq R,
\\[5pt]
& \begin{bmatrix} W_1 & M\\ M^\top & W_2 \end{bmatrix} \succeq 0,
\end{array}
\end{equation}
where $W_1 \in \mathbb{R}^{d_1 \times d_1}$, $W_2 \in \mathbb{R}^{d_2 \times d_2}$, and $M \in \mathbb{R}^{d_1 \times d_2}$. 
By using the characterization of semidefiniteness via the Schur complement, \eqref{eq:mix_norm_estimator_SDP} can be equivalently rewritten as a standard form of SDP given by
\begin{equation}
\label{eq:mix_norm_estimator_SDP_std}
\begin{array}{ll}
\displaystyle \mathop{\text{minimize}}_{W_1,W_2,M,t} & t \\
\text{subject to} & \mathrm{trace}(W_1) \leq R, \\[5pt]
& \begin{bmatrix} I_{Ld_2} & \mathcal{A}(M) \\ \left(\mathcal{A}(M)\right)^\top & 2 y^\top \mathcal{A}(M) - y^\top y + t \end{bmatrix} \succeq 0 \\[10pt]
& \begin{bmatrix} W_1 & M\\ M^\top & W_2 \end{bmatrix} \succeq 0, \ \begin{bmatrix} 1 & \mathrm{diag}(W_2) \\ \mathrm{diag}(W_2)^\top & R \end{bmatrix} \succeq 0, \\[10pt]
& \begin{bmatrix} 1 & M e_k \\ e_k^\top M^\top & \alpha \end{bmatrix} \succeq 0, \quad \forall k \in [d_2].
\end{array}
\end{equation}
The linear matrix inequalities (LMI) for semidefinite constraints in \eqref{eq:mix_norm_estimator_SDP_std} can be combined into a single LMI by combining the matrices in LMIs as a block diagonal matrix. 
Then \eqref{eq:mix_norm_estimator_SDP_std} can be solved by general-purpose SDP solvers such as SeDuMi \citep{sturm1999sedumi} and YALMIP \citep{lofberg2004yalmip}.
However, these SDP solvers may not scale well to large instances. 
Instead, we derive an ADMM algorithm \citep{boyd2011distributed} as follows. 
First, we consider an alternative formulation obtained by penalizing the constraints on $W_1$ and $W_2$, which is given by
\begin{equation}
\label{eq:admm_penalized}
\begin{array}{ll}
\displaystyle \mathop{\text{minimize}}_{W_1,W_2,M} & \|y - \mathcal{A}(M)\|_2^2 + \lambda_1 \mathrm{trace}(W_1) + \lambda_2 \|\mathrm{diag}(W_2)\|_\infty  \\
\text{subject to} & \|M\|_{1 \rightarrow 2} \leq \alpha, \\[5pt]
&\begin{bmatrix}
W_1 & M \\
M^\top & W_2
\end{bmatrix} \succeq 0
\end{array}
\end{equation}
for some positive constants $\lambda_1$ and $\lambda_2$. 
Next, to split the penalty terms and constraints, we introduce auxiliary variables $\widetilde{W}_1$, $\widetilde{W}_2$, and $\widetilde{M}$, which leads to an equivalent reformulation of \eqref{eq:admm_penalized} in the form of 
\begin{equation}
\label{eq:Lagrange_formulatio}
\begin{array}{ll}
\displaystyle \mathop{\text{minimize}}_{W_1,W_2,M,\tilde{W}_1,\widetilde{W}_2,\widetilde{M}} & \|y - \mathcal{A}(M)\|_2^2 + \lambda_1 \mathrm{trace}(\widetilde{W}_1) + \lambda_2 \|\mathrm{diag}(W_2)\|_\infty  \\
\text{subject to} & \|M\|_{1 \rightarrow 2} \leq \alpha, \ M=\widetilde{M}, \ W_1 = \widetilde{W}_1, \ W_2 = \widetilde{W}_2, \\[5pt]
&\begin{bmatrix}
\widetilde{W}_1 & \widetilde{M} \\
\widetilde{M}^\top & \widetilde{W}_2
\end{bmatrix} \succeq 0.
\end{array}
\end{equation}
Then, the augmented Lagrangian of \eqref{eq:Lagrange_formulatio} with respect to the equality constraints is written as
\begin{equation}
\begin{aligned}
L(M,W_1,W_2,\widetilde{M},\widetilde{W}_1,\widetilde{W}_2,Z) 
&= \|y - \mathcal{A}(M)\|_2^2 + \lambda_1 \mathrm{trace}(\widetilde{W}_{1}) + \lambda_2 \|\mathrm{diag}(W_2)\|_\infty \\
&\quad + \left\langle Z, \begin{bmatrix} 
W_1 & M \\ M^\top & W_2
\end{bmatrix} 
- 
\begin{bmatrix} 
\widetilde{W}_1 & \widetilde{M} \\ \widetilde{M}^\top & \widetilde{W}_2
\end{bmatrix}  \right\rangle
+ \frac{\rho}{2} 
\left\|
\begin{bmatrix} 
W_1 & M \\ M^\top & W_2
\end{bmatrix} 
- 
\begin{bmatrix} 
\widetilde{W}_1 & \widetilde{M} \\ \widetilde{M}^\top & \widetilde{W}_2
\end{bmatrix} 
\right\|_\mathrm{F}^2,
\end{aligned}
\end{equation}
where its domain is determined by
\[
\|M\|_{1 \rightarrow 2} \leq \alpha
\quad \text{and} \quad 
\begin{bmatrix}
\widetilde{W}_1 & \widetilde{M} \\
\widetilde{M}^\top & \widetilde{W}_2
\end{bmatrix} \succeq 0.
\]

The ADMM algorithm, outlined in Algorithm \ref{algo:ADMM}, is obtained by minimizing the augmented Lagrangian with respect to $\{M, W_1, W_2\}$ and $\{\widetilde{M}, \widetilde{W}_1, \widetilde{W}_2\}$ separately, followed by the linear update of the dual variable $Z$. 
\begin{algorithm}[t]
\caption{ADMM algorithm for mixed-norm-based estimator}
\label{algo:ADMM}
\begin{algorithmic}
\STATE{\bfseries Input:} $\widetilde{M}^0, \widetilde{W}_1^0, \widetilde{W}_2^0, Z^0$
\STATE{\bfseries Initialize:} $t = 0$ 
\WHILE{not converged}
\STATE $(M^{t+1}, W_1^{t+1}, W_2^{t+1}) = \displaystyle \mathop{\text{argmin}}_{M,W_1,W_2} \ \left\{ L(M,W_1,W_2,\widetilde{M}^{t},\widetilde{W}_1^{t},\widetilde{W}_2^{t},Z^{t}) : \norm{M}_{1 \rightarrow 2} \leq \alpha \right\}$
\STATE $(\widetilde{M}^{t+1}, \widetilde{W}_1^{t+1}, \widetilde{W}_2^{t+1}) = \displaystyle \mathop{\text{argmin}}_{\widetilde{M},\widetilde{W}_1,\widetilde{W}_2} \ \left\{ L(M^{t+1},W_1^{t+1},W_2^{t+1},\widetilde{M},\widetilde{W}_1,\widetilde{W}_2,Z^{t}) : \begin{bmatrix} \widetilde{W}_1 & \widetilde{M} \\ \widetilde{M}^\top & \widetilde{W}_2 \end{bmatrix} \succeq 0 \right\}$
\STATE $\displaystyle Z^{t+1} = Z^t  +\rho\left( \begin{bmatrix} \widetilde{W}_1^{t+1} & \widetilde{M}^{t+1} \\ (\widetilde{M}^{t+1})^\top & \widetilde{W}_2^{t+1} \end{bmatrix} - \begin{bmatrix} W_1^{t+1} & M^{t+1} \\ (M^{t+1})^\top & W_2^{t+1} \end{bmatrix} \right)$
\STATE $t \leftarrow t+1$
\ENDWHILE
\STATE{\bfseries Output:} $\widetilde{M}^t$
\end{algorithmic}
\end{algorithm}
As we will show below, splitting variables in the above way makes the resulting sub-problems easy to solve, which enables efficient updates in the ADMM algorithm. \\

\noindent\noindent\textbf{Update of $\{M, W_1, W_2\}$:} This step is done by the following four separate sub-problems. 
To explain the solutions to these sub-problems, we introduce submatrices of $Z^t$ denoted by $Z_{11}^t \in \mathbb{R}^{d_1 \times d_1}$, $Z_{12}^t \in \mathbb{R}^{d_1 \times d_2}$, and $Z_{22}^t \in \mathbb{R}^{d_2 \times d_2}$ such that 
\[
Z^t = \begin{bmatrix}
Z_{11}^t & Z_{12}^t \\
\left(Z_{12}^t\right)^\top & Z_{22}^t
\end{bmatrix}.
\]
First note that the determination of the optimal $M$ is decoupled from the other variables $W_1$ and $W_2$, which is given by
\begin{equation}
\label{eq:admm_update_M}
M^{t+1} = \mathop{\mathrm{argmin}}_{ \norm{M}_{1\to 2} \leq \alpha} \ \|y - \mathcal{A}(M)\|_2^2  + \langle Z_{12}^{t}, \widetilde{M}^{t} - M \rangle +\frac{\rho}{2} \|\widetilde{M}^{t} - M\|_\mathrm{F}^2.
\end{equation}
Indeed, each column of $M^{t+1}$ is updated independently from the other columns. 
Note that the measurement vector $y \in \mathbb{R}^n$ is decomposed into $d_2$ non-overlapping blocks with respect to the column from which the measurements are taken from. Recall that our measurements are given as $y_k = \mathrm{tr}(A_k^\top M_0) + \eta_k$ for $k \in [Ld_2]$. 
For each $l \in [d_2]$, we denote the column vector consisting of all measurements with $A_k = \xi_k e_{j_k}^\top$ such that $j_k = l$ by $\breve{y}_l \in \mathbb{R}^L$. 
Then $y$ is equivalent to $[\breve{y}_1^\top, \breve{y}_2^\top, \dots, \breve{y}_{d_2}^\top]^\top$ up to a permutation. We can construct matrices $\breve{A}_1, \dots, \breve{A}_{d_2} \in \mathbb{R}^{L \times d_1}$ such that the noise-free version of $\breve{y}_l$ is given by $\breve{y}_l = \breve{A}_l (M_0 e_l)$. 
Then each summand in the objective function in \eqref{eq:admm_update_M} is decomposed as
\begin{align*}
\|y - \mathcal{A}(M)\|_2^2 &= \sum_{l=1}^{d_2} \|\breve{y}_l - \breve{A}_l(M e_l)\|_2^2, \\
\langle Z_{12}^{k}, \widetilde{M}^{t} - M \rangle
&= \sum_{l=1}^{d_2} \langle Z_{12}^{t} e_l, \widetilde{M}^{t} e_l - M e_l \rangle, \\
\frac{\rho}{2} \big\|\widetilde{M}^t - M\big\|_\mathrm{F}^2 
&= \sum_{l=1}^{d_2} \frac{\rho}{2} \big\|\widetilde{M}^t e_l - M e_l\big\|_2^2.
\end{align*}
Furthermore, the constraint $\norm{M}_{1 \to 2} \leq \alpha$ is equivalently rewritten as 
\[
\norm{M e_l}_2 \leq \alpha, \quad \forall l \in [d_2].
\]
Therefore, the optimization in \eqref{eq:admm_update_M} is rewritten as a set of decoupled optimization problems over columns given by
\begin{equation}
\label{eq:admm_update_ml}
M^{t+1} e_l = \mathop{\mathrm{argmin}}_{w \in \mathbb{R}^{d_1} : \norm{w}_2 \leq \alpha} \ \|y_l - \breve{A}_l w\|_2^2  + e_l^\top (Z_{12}^{t})^\top (\widetilde{M}^{t} e_l - w) +\frac{\rho}{2} \|\widetilde{M}^{t} e_l - w\|_2^2.
\end{equation}
The optimization problem in \eqref{eq:admm_update_ml} is rewritten as
\begin{equation}
\label{eq:admm_update_ml2}
\min_{\|w\|_2 \leq \alpha} \|b - Q w\|_2^2
\end{equation}
with $Q = \breve{A}_l^\top \breve{A}_l + \frac{\rho}{2} I_{d_1}$ and $b = \breve{A}_l^\top y - \frac{1}{2} Z_{12}^t e_l - \frac{\rho}{2} \widetilde{M}^t e_l$. 
This is a norm-constrained least squares problem, which is equivalent to a ridge-regression problem via the Lagrangian formulation, i.e. there exists $\lambda > 0$ such that \eqref{eq:admm_update_ml2} is equivalent to
\[
L(w,\lambda) = \|b - Qw\|_2^2 + \lambda (\|w\|_2^2 - \alpha^2).
\]
This can be solved efficiently using least squares and a binary search over $\lambda$. Further details are provided in Appendix~\ref{sec:app_least_squares_l2_constraint}.\footnote{The norm-constrained least squares in \eqref{eq:admm_update_ml2} could be avoided as one introduce another block to the ADMM algorithm. However, the resulting ADMM algorithm with three blocks is not necessarily convergent \citep{chen2016direct}.
Next, $W_1$ is updated by a simple closed form given by 
\[
W_{1}^{t+1} = \widetilde{W}_1^{t} - \rho^{-1} Z_{11}^{t}.
\]
The off-diagonal entries of $W_2$ are updated similarly by
\[
\big[ W_{2}^{t+1} \big]_{i,j} = \big[ \widetilde{W}_{2}^{t} \big]_{i,j} - \rho^{-1} \big[ Z_{22}^{t} \big]_{i,j}, \quad \forall i \neq j \in [d_2].
\]
Finally, the diagonal entries of $W_2$ are updated by
\begin{equation}
\label{eq:admm_update_w2d}
\mathop{\text{diag}}(W_{2}^{t+1}) = \mathop{\mathrm{argmin}}_{u \in \R^{d_2}} \ \lambda_2 \norm{u}_\infty + \frac{\rho}{2} \norm{u - \mathrm{diag}(\widetilde{W}_2^t + \rho^{-1} Z_{22}^t)}_2^2.
\end{equation}
Note that the right-hand side of \eqref{eq:admm_update_w2d} corresponds to the proximal operator of $\| \cdot \|_\infty$. It has been shown \citep{fang2018max} that there exists a closed form expression for the solution to \eqref{eq:admm_update_w2d}.} \\

\noindent\noindent\textbf{Update of $\{\widetilde{M}, \widetilde{W}_1, \widetilde{W}_2\}$:} 
With a shorthand notation 
\[
\Phi = \begin{bmatrix} \tilde{W}_1 & \tilde{M} \\ \tilde{M}^\top & \tilde{W}_2 \end{bmatrix},
\]
the objective function is written as a quadratic function in $\Phi$ whose Hessian is a scaled identity matrix. Therefore, the solution is written as 
\begin{align*}
\begin{bmatrix} \widetilde{W}_1^{t+1} & \widetilde{M}^{t+1} \\ (\widetilde{M}^{t+1})^\top & \widetilde{W}_2^{t+1} \end{bmatrix}
& = \Pi_{\mathcal{S}_+^d}\left(\Phi^{k+1} - \rho^{-1} \left( Z^{k} + \lambda_1 \begin{bmatrix}
I_{d_1} & 0_{d_1,d_2} \\ 
0_{d_2,d_1} & 0_{d_2} \end{bmatrix} \right )\right),
\end{align*} 
where $\Pi_{\mathcal{S}_+^d}(\cdot)$ denotes the orthogonal projection operator onto the set of $d$-by-$d$ positive semidefinite matrices $\mathcal{S}_+^d$.

\section{Relation among tensor norms of rank-$r$ operators}
\label{sec:rankineq}

Recall that our primary goal is to establish a unifying analysis framework for low-rank matrix recovery from localized measurements. 
This is achieved through 
the entropy analysis between the projective and injective tensor products, which is derived in the next section. 
To this end, in this section, we verify that the max norm and the mixed norm are equivalent to the projective norm on the corresponding tensor product. 

The equivalence of the $\gamma_2$ norm and the projective norm on $\ell_\infty^n \otimes \ell_\infty^m$ is known as a form of the little Grothendieck's theorem (e.g., \citep[Theorem~3.1]{pisier2012grothendieck}), stated as the following lemma. 
\begin{lemma}[{little Grothendieck \citep[Theorem~3.1]{pisier2012grothendieck}}]
\label{lemma:nu1gamma2}
Let $T \in \ell_\infty^n \otimes \ell_\infty^m$. Then
\begin{equation}
\label{eq:littleGT}
\gamma_2(T) \leq \nu_1(T) \leq K_\mathrm{G} \gamma_2(T),
\end{equation}
where $K_G$ denotes the Grothendieck constant that satisfies $1.67 \leq K_G \leq 1.79$.
\end{lemma}

\noindent Indeed, one obtains a more popular form of Grothendieck's theorem by plugging in the expression of the projective norm (or equivalently $1$-nuclear norm since $T$ is finite-rank) in \eqref{eq:projnorm_infinf} and the $\gamma_2$ norm in \eqref{eq:gamma2_infinf_littlgGT} into the inequality in \eqref{eq:littleGT}. 
Furthermore, the following corollary is given as a direct consequence of Lemmas~\ref{lemma:linial} and \ref{lemma:nu1gamma2}.
\begin{cor}
\label{cor:injproj_infinf}
Let $T \in \ell_\infty^n \otimes \ell_\infty^m$ satisfy $\mathrm{rank}(T) \leq r$. Then
\[
\|T\| \leq \nu_1(T) \leq K_G \sqrt{r} \|T\|.
\]
\end{cor}
\noindent Corollary~\ref{cor:injproj_infinf} shows the equivalence of the injective and projective tensor norms of a rank-$r$ linear operator in $\ell_\infty^n \otimes \ell_\infty^m$ up to $K_G \sqrt{r}$.

Next, we present the relation between the 2-summing norm and the 1-nuclear norm of a linear operator in $\ell_\infty^n \otimes \ell_2^m$. 
Indeed, we show the norm equivalence in a slightly more general setting, in which the tensor product is given as $\ell_\infty^n \otimes Y$ for a Banach space $Y$ of type-2 defined as follows.
\begin{defi}
A Banach space $X$ has {\em type $p$} if there exists a constant $C$ such that for all finite sequence $(x_j)$ in $X$
\begin{equation}
\label{eq:def:typep}
\Big(\mathbb{E} \Big\|\sum_j \epsilon_j x_j\Big\|_X^p\Big)^{1/p} \leq C \Big(\sum_j \|x_j\|_X^p\Big)^{1/p},
\end{equation}
where $(\epsilon_j)$ is a Rademacher sequence \citep{pisier1999volume}.
The type-$p$ constant of $X$, denoted by $\tau_p(X)$, is the smallest constant $C$ that satisfies \eqref{eq:def:typep}.
\end{defi}

\noindent The triangle inequality implies that every normed space is of type 1 with constant $\tau_1(X) = 1$. In this paper, we are interested in Banach spaces of type 2. For example, $L_p(\Omega,\Sigma,\mu)$ satisfies
\[
\tau_{\min(p,2)}(L_p) \leq \sqrt{p}
\]
for $1 \leq p < \infty$ (e.g. \cite[Lemma~3]{carl1985inequalities}). This implies that $\ell_p^n$ is of type 2 for all $2 \leq p < \infty$.

The following lemma shows that the 1-nuclear norm of $T \in \ell_\infty^n \otimes Y$ is equivalent to the 2-summing norm of the adjoint $T^*$ up to the type-2 constant of $Y$. The proof of Lemma~\ref{lemma:pi2a_nu1} is provided in Appendix~\ref{sec:proof:lemma:pi2a_nu1}.

\begin{lemma}
\label{lemma:pi2a_nu1}
Let $T \in \ell_\infty^n \otimes Y$, where $Y$ is a finite-dimensional Banach space of type-2. 
Suppose that $T^*$ is 2-summing. 
Then we have
\[
\pi_2(T^*) \leq \nu_1(T) \leq \sqrt{2} \, \tau_2(Y) \pi_2(T^*).
\]
\end{lemma}

By combining Lemmas~\ref{lemma:pi2a_op} and \ref{lemma:pi2a_nu1} with the fact that $\tau_2(\ell_p^m) \leq \sqrt{p}$ for $2 \leq p < \infty$ (e.g., see \cite[Lemma~3]{carl1985inequalities}), we obtain the following corollary.

\begin{cor}
\label{cor:injproj_inftype2}
Let $T \in \ell_\infty^n \otimes \ell_2^m$ be of rank-$r$. 
Then, we have
\[
\|T\| \leq \pi_2(T^*) \leq \nu_1(T) \leq \sqrt{2r} \|T\|.
\]
\end{cor}

\noindent Corollary~\ref{cor:injproj_inftype2} implies that both the 2-summing and 1-nuclear norms of a rank-$r$ operator are equivalent to the operator norm up to $\sqrt{2r}$. 
Moreover, the 1-nuclear and operator norms correspond to the projective and injective tensor norms in this setting. 
In the next section, we will use the interlacing property of these norms to compute entropy integrals, which arise in the analysis of relevant low-rank recovery problems.

\section{Entropy estimates of tensor products}
\label{sec:ent_est}

This section is devoted to derive estimates on covering number and the resulting entropy integral for tensor products of Banach spaces. 
Let us first recall relevant definitions to state the main results. 

For symmetric convex bodies $D$ and $E$, the \emph{covering number} $N(D,E)$ and the \emph{packing number} $M(D,E)$ are respectively defined by
\begin{align*}
N(D,E) & := \min \Big\{ l : \exists y_1,\dots,y_l \in D, \, D \subset \bigcup_{1\le j \le l} (y_j + E) \Big\}, \\
M(D,E) & := \max \Big\{ l : \exists y_1,\dots,y_l \in D, \, y_j-y_k \not\in E, \, \forall j\neq k \Big\}.
\end{align*}
Indeed, they are related to each other by
\[
N(D,E) \leq M(D,E) \leq N(D,E/2) \, .
\]
For $T \in L(X,Y)$, the \emph{dyadic entropy number} \citep{carl1990entropy} is defined by
\[
e_k(T) := \inf\{\epsilon > 0 : M(T(B_X),\epsilon B_Y)\leq 2^{k-1}\}.
\]
We will use the following shorthand notation for the weighted summation of the dyadic entropy numbers:
\[
\mathcal{E}_{2,1}(T) := \sum_{k=0}^\infty \frac{e_k(T)}{\sqrt{k}},
\]
which is up to a constant equivalent to the entropy integral $\int_0^{\infty} \sqrt{\ln N(T(B_X),\epsilon B_Y)} d\epsilon$ \citep{pisier1999volume}, which plays a key role in analyzing properties on random linear operators on low-rank matrices.

We derive the $\mathcal{E}_{2,1}$ of the identity operator from the injective tensor product to the projective tensor product of a set of Banach space pairs. The main machinery in deriving these estimates is Maurey's empirical method \citep{carl1985inequalities}, summarized in the following lemma.

\begin{lemma}[{\cite[Lemma~3.4]{junge2020generalized}}]\label{mm}
Let $T \in L(\ell_1^n,\ell_\infty^m(\ell_2^d))$. Then
\[
\mathcal{E}_{2,1}(T) \leq C \sqrt{1+\ln n} \, (1+\ln m)^{3/2} \|T\|.
\]
\end{lemma}

In order to apply Lemma~\ref{mm} to $\ell_\infty^n \otimes \ell_\infty^m$, we use the fact that $\ell_\infty^m \mathbin{\widecheck{\otimes}} \ell_\infty^n$ is isometrically isomorphic to $\ell_\infty^{mn}$. In fact,
\[
\|M\|_{\ell_\infty^m \mathbin{\widecheck{\otimes}} \ell_\infty^n}
= \max_{1\leq j\leq n} \|M e_j\|_\infty = \|\mathrm{vec}(M)\|_\infty,
\]
where $\mathrm{vec}(M)$ rearranges $M \in \mathbb{R}^{m \times n}$ into $\mathbb{R}^{mn}$ by stacking its columns vertically. Furthermore, the trace dual and the Banach space dual of $\ell_1^n \mathbin{\widehat{\otimes}} \ell_1^m$ are $\ell_1^m \mathbin{\widecheck{\otimes}} \ell_1^n$ and $\ell_\infty^n \mathbin{\widecheck{\otimes}} \ell_\infty^m$, respectively. Therefore, it follows that $\ell_1^m \mathbin{\widehat{\otimes}} \ell_1^n$ is isometrically isomorphic to $\ell_1^{mn}$. 
With these isometric isomorphisms, Maurey's empirical method in Lemma~\ref{mm} provides the following estimate.

\begin{prop}
\label{prop:e21_pi2e_ellinfellinf}
There exists a numerical constant $C$ such that
\begin{align*}
\mathcal{E}_{2,1}(\mathrm{id}: \ell_\infty^m \mathbin{\widehat{\otimes}} \ell_\infty^n \to \ell_\infty^m \mathbin{\widecheck{\otimes}} \ell_\infty^n) 
\leq C \sqrt{1+m+n} \, (1+\ln m + \ln n)^{3/2}.
\end{align*}
\end{prop}

\begin{proof}[Proof of Proposition~\ref{prop:e21_pi2e_ellinfellinf}]
\newcommand\isom{\mathrel{\stackon[-0.1ex]{\makebox*{\scalebox{1.08}{\AC}}{=\hfill\llap{=}}}{{\AC}}}}
\newcommand\nvisom{\rotatebox[origin=cc] {-90}{$ \isom $}}
\newcommand\visom{\rotatebox[origin=cc] {90} {$ \isom $}}
Let $\iota_m: \ell_1^m \to \ell_\infty^{2^m}$ be defined by
\[
\iota_m\left((a_j)_{j=1}^m\right) = \left(\sum_{j=1}^m \epsilon_j a_j\right)_{(\epsilon_j)_{j=1}^m \in \{\pm 1\}^m}.
\]
Then we have
\begin{align*}
\left\|\iota_m((a_j)_{j=1}^m)\right\|_\infty
= \max \left\{ \bigg| \sum_{j=1}^m \epsilon_j a_j \bigg| : (\epsilon_j)_{j=1}^m \in \{\pm 1\}^m \right\} 
= \left\|(a_j)_{j=1}^m\right\|_1.
\end{align*}
This implies that $\ell_1^m$ is isometrically isomorphic to $E_m = \iota_m(\ell_1^m) \subset \ell_\infty^{2^m}$. Since $\iota_m \in L(\ell_1^m,E_m)$ is an isomorphism, there exists $J_m \in L(\ell_\infty^{2^m},\ell_1^m)$ such that $\iota_m J_m|_{E_m}$ (resp. on $J_m \iota_m$) is the identity on $E_m$ (resp. $\ell_1^m$). Then $J_m^*(\ell_\infty^m) = E_m^*$ and $\ell_\infty^m$ is isometrically isomorphic to $E_m^*$. Furthermore, by the Hahn-Banach theorem, $\iota_m^* \in L(\ell_1^{2^m},\ell_\infty^m)$ is surjective and isometric. It also follows that $J_m$ is also surjective and isometric.

Then there exists an isometry $\iota$ that embeds $\ell_1^n \mathbin{\widecheck{\otimes}} \ell_1^m$ into $\ell_\infty^{2^n} \mathbin{\widecheck{\otimes}} \ell_\infty^{2^m}$. Indeed, $\ell_1^n \mathbin{\widecheck{\otimes}} \ell_1^m$ is identified to $L(\ell_\infty^n, \ell_1^m)$. By $\iota_m \in L(\ell_1^m,\ell_\infty^{2^m})$ and $J_n^* \in L(\ell_\infty^n,\ell_1^{2^n})$, $L(\ell_\infty^n, \ell_1^m)$ is isometrically isomorphic to $L(E_n^*,E_m)$. Since $E_n^*$ and $E_m$ are subspaces of $\ell_1^{2^n}$ and $\ell_\infty^{2^m}$ respectively, $L(E_n^*,E_m)$ embeds into $L(\ell_1^{2^n}, \ell_\infty^{2^m})$, which is identified to $\ell_\infty^{2^n} \mathbin{\widecheck{\otimes}} \ell_\infty^{2^m}$. Furthermore, similarly to the above argument, by the Hahn-Banach theorem, $\iota^*: \ell_1^{2^n} \mathbin{\widehat{\otimes}} \ell_1^{2^m} \to \ell_\infty^n \mathbin{\widehat{\otimes}} \ell_\infty^m$ is surjective and isometric.

Then we obtain the following commutative diagram:
\begin{center}
\begin{tikzcd}[row sep=large, column sep=large, ar symbol/.style = {draw=none,"#1" description,sloped}, isomorphic/.style = {ar symbol={\cong}}]
\ell_1^{2^n} \mathbin{\widehat{\otimes}} \ell_1^{2^m} \arrow[d, two heads, "\iota^*"] \arrow{dr}{T} \\
\ell_\infty^n \mathbin{\widehat{\otimes}} \ell_\infty^m \arrow[r, "\text{id}"] & \ell_\infty^n \mathbin{\widecheck{\otimes}} \ell_\infty^m
\end{tikzcd}
\end{center}

Since the map $\iota^*$ is surjective and isometric, we have
\begin{align*}
e_k(\mathrm{id}: \ell_\infty^n \mathbin{\widehat{\otimes}} \ell_\infty^m \to \ell_\infty^n \mathbin{\widecheck{\otimes}} \ell_\infty^m)
\leq
e_k(T: \ell_1^{2^n} \mathbin{\widehat{\otimes}} \ell_1^{2^m} \to \ell_\infty^n \mathbin{\widecheck{\otimes}} \ell_\infty^m)
=
e_k(\mathrm{id}: \ell_1^{2^{n+m}} \to \ell_\infty^{nm}).
\end{align*}

Finally, the assertion follows by Maurey's empirical method in Lemma~\ref{mm}.
\end{proof}

\noindent Next, in order to apply Lemma~\ref{mm} to $\ell_\infty^n \otimes \ell_p^m$ with $2 \leq p < \infty$, we use the following result that shows embedding of finite-dimensional $\ell_p$ space to $\ell_1$ up to a small Banach-Mazur distance.

\begin{lemma}\cite[Lemma~5]{carl1985inequalities}
\label{embed}
Let $1 < p \leq 2$. For each $\epsilon > 0$, there exists a constant $c(p,\epsilon) > 0$ such that for each $m$, $\ell_1^m$ contains a subspace $(1+\epsilon)$-isomorphic to $\ell_p^k$ with $k \geq c(p,\epsilon) m$.
\end{lemma}

\noindent Then we can obtain the following entropy estimate for $\ell_\infty^n \otimes \ell_p^m$ with $2 \leq p < \infty$ by combining Lemmas~\ref{mm} and \ref{embed}.

\begin{prop}
We have
\label{prop:e21_pi2e_ellinfell2}
\begin{align*}
\mathcal{E}_{2,1}(\mathrm{id}: \ell_\infty^n \mathbin{\widehat{\otimes}} \ell_2^m \to \ell_\infty^n \mathbin{\widecheck{\otimes}} \ell_2^m)
\leq C \sqrt{1+n+m} \, (1+\ln n)^{3/2}.
\end{align*}
\end{prop}

\begin{proof}[Proof of Proposition~\ref{prop:e21_pi2e_ellinfell2}]
By Lemma~\ref{embed}, there exists $\Upsilon: \ell_2^m \to \ell_1^{\lceil m/c(p,\epsilon)\rceil}$ such that
\begin{align}
\label{eq:embed_Upsilon}
\left| \norm{\Upsilon x}_1 - \norm{x}_2 \right| \leq \epsilon.
\end{align}
Let $E = \Upsilon(\ell_2^m)$ denote the image of $\ell_2^m$ via $\Upsilon$. Define $E^\perp = \{y \in \ell_\infty^{\lceil m/c(p,\epsilon)\rceil} : \langle y, x \rangle = 0, \forall x \in E\}$, where $\langle \cdot, \cdot \rangle$ denotes the canonical bilinear transform on $\ell_\infty^{\lceil m/c(p,\epsilon)\rceil} \times \ell_1^{\lceil m/c(p,\epsilon)\rceil}$ corresponding to the functional evaluation. 
By the Hahn-Banach theorem, $E^*$ is isometrically isomorphic to $\ell_\infty^{\lceil m/c(p,\epsilon)\rceil}/E^\perp$ and there exists an isometric epimorphism $\varrho$ from $\ell_\infty^{\lceil m/c(p,\epsilon)\rceil}$ to the quotient space $\ell_\infty^{\lceil m/c(p,\epsilon)\rceil}/E^\perp$, i.e. $\varrho: \ell_\infty^{\lceil m/c(p,\epsilon)\rceil} \twoheadrightarrow \ell_\infty^{\lceil m/c(p,\epsilon)\rceil}/E^\perp \cong E^*$. It follows from \eqref{eq:embed_Upsilon} that 
\begin{equation*}
\left| \norm{\Upsilon^*|_{E^*} y}_2 - \norm{y}_{E^*} \right| \leq \delta \norm{y}_{E^*}, \quad \forall y \in E^*. 
\end{equation*}
Therefore, $\Upsilon^*|_{E^*} \circ \varrho: Y^* \twoheadrightarrow \ell_2^m$ is an epimorphism satisfying $\norm{\Upsilon^*|_{E^*} \circ \varrho} \leq 1+\epsilon$. 
Furthermore, by the construction in the proof of \ref{prop:e21_pi2e_ellinfellinf}, there is an isometric epimorphism from $\ell_1^{2^{\lceil m/c(p,\epsilon)\rceil}}$ to $\ell_\infty^{\lceil m/c(p,\epsilon)\rceil}$. 
Therefore, we have a surjective map from $\ell_1^{2^{\lceil m/c(p,\epsilon)\rceil}}$ to $\ell_2^m$. 
On the other hand, we have a surjective map from $\ell_1^{2^n}$ to $\ell_\infty^n$, which is indeed an isometric epimorphism. 
By the surjectivity of the projective tensor product, we obtain that there is a surjective map from $\ell_1^{2^n} \mathbin{\widehat{\otimes}} \ell_1^{2^{\lceil m/c(p,\epsilon)\rceil}}$ to $\ell_\infty^n \mathbin{\widehat{\otimes}} \ell_2^m$. 
Next note that $\ell_1^{2^n} \mathbin{\widehat{\otimes}} \ell_1^{2^{\lceil m/c(p,\epsilon)\rceil}}$ is isometrically identified to $\ell_1^{2^{n+\lceil m/c(p,\epsilon)\rceil}}$. 
Furthermore, $\ell_\infty^n \mathbin{\widecheck{\otimes}} \ell_2^m$ is isometrically identified to $\ell_\infty^n(\ell_2^m)$. 
By combining the above results, we obtain the following commutative diagram. 
\begin{center}
\begin{tikzcd}[row sep=large, column sep=large, ar symbol/.style = {draw=none,"#1" description,sloped}, isomorphic/.style = {ar symbol={\cong}}]
\ell_1^{2^{n+\lceil m/c(p,\epsilon)\rceil}} \arrow[d, two heads] 
\arrow{dr}{} \\
\ell_\infty^n \mathbin{\widehat{\otimes}} \ell_2^m
\arrow{r}{\mathrm{id}}
& \ell_\infty^n(\ell_2^m)
\end{tikzcd}
\end{center}
Then the assertion follows from Lemma~\ref{mm} due to the surjectivity of the entropy number \cite[p. 12]{carl2008entropy}. 
\end{proof}

\section{Estimation error upper bounds}
\label{sec:err_bnd}

In this section, we derive statistical error bounds for decentralized subspace sketching and matrix completion by using Proposition~\ref{prop:deterministic_errbnd}. 
We obtain tail bounds of the geometric quantities $\theta$ and $\Gamma$. 

\subsection{Decentralized subspace sketching}
\label{sec:DS}

We first prove Theorem~\ref{thm:eb_ds} for decentralized subspace sketching.
The following lemma, as a corollary of Proposition~\ref{prop:e21_pi2e_ellinfell2}, provides a tail bound on $\theta$. 
\begin{lemma}
\label{lemma:rip_dc}
Let $(b_k)_{k=1}^{Ld_2}$, $(\eta_k)_{k=1}^{Ld_2}$, and $\tnorm{\cdot}$ be as in Theorem~\ref{thm:eb_ds}.
Then there exists an absolute constant $C$ such that
\begin{align*}
& \sup_{\vert\vert\vert M \vert\vert\vert \leq 1}
\left| \norm{M}_\mathrm{F}^2 - \frac{1}{L} \sum_{k=1}^{Ld_2} |\langle b_k e_{j_k}^\top, M \rangle|^2 \right| \\
& \quad \leq
C d_2 \left( \frac{r(d_1+d_2)\log^3(d_1+d_2) \vee \log(2\zeta^{-1})}{Ld_2} \vee \sqrt{\frac{r(d_1+d_2)\log^3(d_1+d_2) \vee \log(2\zeta^{-1})}{Ld_2}} \right)
\end{align*}
holds with probability $1-\zeta$.
\end{lemma}

\begin{proof}
Let $Q_M \in \mathbb{R}^{Ld_1d_2 \times Ld_1d_2}$ be a block diagonal matrix defined by
\begin{equation*}
Q_{M} := \frac{1}{\sqrt{L}}
\begin{bmatrix}
I_L \otimes (M e_1)^{\top} & 0 & \cdots & 0 \\
0 & I_L \otimes (M e_2)^{\top} & \cdots & 0 \\
 \vdots & \vdots & \ddots & \vdots \\
0 & 0 & \cdots & I_L \otimes (M e_{d_2})^{\top}
\end{bmatrix}. 
\end{equation*} 
Then we have
\begin{equation*}
\frac{1}{L} \sum_{k=1}^{Ld_2} |\langle b_k e_{j_k}^\top, M \rangle |^2  = \norm{Q_M \xi}^2
\end{equation*} 
for $\xi \sim \mathcal{N}(0,I_{Ld_1d_2})$.
It follows that $\norm{Q_M \xi}^2$ is a Gaussian chaos indexed by $M$. 
Furthermore, it satisfies
\begin{equation*}
     \mathbb{E} \norm{Q_M \xi}^2 = \norm{M}_\mathrm{F}^2.
\end{equation*}
Therefore, it suffices to obtain a tail estimate on the supremum of $|\norm{Q_M \xi}_2^2 - \mathbb{E} \norm{Q_M \xi}_2^2|$ over $\{M: \tnorm{M} \leq 1\}$. 
To this end, we use the concentration inequality by \citet{krahmer2014suprema}, which is summarized in the following theorem.

\begin{theorem}[{Theorem~3.1 in \citep{krahmer2014suprema}}]
\label{thm:kmr}
Let $\xi \in \mathbb{R}^n$ be a Gaussian vector with $\mathbb{E}[\xi] = 0$ and $\mathbb{E}[\xi \xi^\top] = I_n$. 
Let $\Delta \subset \mathbb{R}^{m \times n}$. 
There exists an absolute constant $C$ such that
\begin{align*}
 \sup_{Q \in \Delta} \left|\norm{Q \xi}_2^2 - \E[\norm{Q \xi}_2^2] \right| 
& \leq C\left( E + V \sqrt{\log(2\zeta^{-1})} + U \log(2\zeta^{-1}) \right)
\end{align*}
holds with probability $1-\zeta$, where
\begin{align*}
E {} & := \tilde{\gamma}_2(\Delta)
\left[ \tilde{\gamma}_2(\Delta) + d_{\mathrm{F}}(\Delta) \right], \\
V {} & := d_{\mathrm{S}}(\Delta) \left[ \tilde{\gamma}_2(\Delta) + d_{\mathrm{F}}(\Delta) \right], \\
U {} & := d_{\mathrm{S}}^2(\Delta).
\end{align*}
Here $\tilde{\gamma}_2(\Delta,\norm{\cdot})$ denotes the Talagrand $\gamma_2$ functional of the metric space given by the spectral norm, and $d_\mathrm{S}(\Delta)$ and $d_\mathrm{F}(\Delta)$ denotes the radii of $\Delta$ with respect to the spectral norm and the Frobenius norm, respectively.\footnote{We used an non-standard notation $\tilde{\gamma}_2$ to distinguish it from the $\gamma_2$ factorization norm in Section~\ref{sec:prelim}.}
\end{theorem}

We apply Theorem~\ref{thm:kmr} to the set $\Delta := \{ Q_M : M \in B_{\vert\vert\vert\cdot\vert\vert\vert}\}$, where the unit ball $B_{\vert\vert\vert\cdot\vert\vert\vert} := \{M : \tnorm{M} \leq 1\}$ is written as $B_{\vert\vert\vert\cdot\vert\vert\vert} = \{M : \norm{M}_{1\to 2} \leq 1, \norm{M}_\mathrm{mixed} \leq \sqrt{r} \}$. 
Note that $d_\mathrm{F}(\Delta)$ and $d_\mathrm{S}(\Delta)$ are respectively upper-bounded by
\begin{align*}
    d_\mathrm{F}(\Delta) \leq \sqrt{d_2}
\quad \text{and} \quad
    d_\mathrm{S}(\Delta) \leq \frac{1}{\sqrt{L}}.
\end{align*}
Let $B_\mathrm{S}$ denote the unit ball with respect to the spectral norm. 
Then, by Dudley's inequality, $\tilde{\gamma}_2(\Delta)$ is upper-bounded through by
\begin{align*}
    \tilde{\gamma}_2 (\Delta) 
    & \lesssim \int_0^{\infty} \sqrt{\log N \left(\Delta, \eta B_\mathrm{S} \right)} \, d\eta \\
    & \leq \frac{1}{\sqrt{L}} \int_0^{\infty} \sqrt{\log N\left ( B_{\vert\vert\vert\cdot\vert\vert\vert}, \eta B_{1\rightarrow 2}\right )} \, d\eta \\
    & \lesssim \frac{\sqrt{r(d_1+d_2)\log^3(d_1+d_2)}}{\sqrt{L}},
\end{align*}
where the second inequality holds since $\norm{Q_M - Q_{M'}} = L^{-1/2} \norm{M-M'}_{1\to 2}$ and the last inequality follows from Proposition~\ref{prop:e21_pi2e_ellinfell2}.
Then $E$, $U$, and $V$ in Theorem~\ref{thm:kmr} are upper-bounded respectively by
\begin{align*}
E & \lesssim 
d_2 \left( \sqrt{\frac{r(d_1+d_2)\log^3(d_1+d_2)}{Ld_2}} + \frac{r(d_1+d_2)\log^{3}(d_1+d_2)}{Ld_2} \right), 
\quad U \leq \frac{d_2}{L d_2},
\intertext{and}
V & \lesssim \frac{d_2}{\sqrt{Ld_2}}
\left ( \frac{\sqrt{r(d_1+d_2)\log^3(d_1+d_2)}}{\sqrt{Ld_2}} + 1\right ).
\end{align*}
Then the assertion follows by plugging in these upper estimates to Theorem~\ref{thm:kmr}. 
\end{proof}

\noindent Next, we obtain a tail bound on $\Gamma$ by the following lemma. 
\begin{lemma}
\label{lemma:gc_dc}
Let $(b_k)_{k=1}^{Ld_2}$, $(\eta_k)_{k=1}^{Ld_2}$, and $\tnorm{\cdot}$ be as in Theorem~\ref{thm:eb_ds}.
Then there exists an absolute constant $C$ such that
\[
\mathbb{E}_{(g_k)} \tnorm{ \sum_{k=1}^{Ld_2} g_k b_k e_{j_k}^\top}_* 
\leq C \sqrt{\log(\zeta^{-1})} \, \sqrt{L d_2} \sqrt{r(d_1+d_2)} \log^{3/2}(d_1+d_2) \sqrt{\log (Ld_2 + 1)}
\]
holds with probability $1-\zeta$. 
\end{lemma}

\begin{proof}
The left-hand side is equivalent (up to a logarithmic factor of the size of the summation) to the corresponding Rademacher complexity (e.g. \citep{ledoux2013probability}, Equation (4.9)), i.e.
\begin{equation*}
\mathbb{E}_{(g_k)} \tnorm{ \sum_{k=1}^{Ld_2} g_k b_k e_{j_k}^\top}_*
\lesssim \sqrt{\log (Ld_2 + 1)}
\mathbb{E}_{(r_k)} \tnorm{ \sum_{k=1}^{Ld_2} r_k b_k e_{j_k}^\top}_*,
\end{equation*}
where $(r_k)_{k=1}^{Ld_2}$ is a Rademacher sequence and the expectation is conditioned on $(b_k)_{k=1}^{Ld_2}$. Then by the symmetry of the standard Gaussian distribution, we obtain
\begin{equation*}
\begin{aligned}
\E_{(r_k)} \tnorm{\sum_{k=1}^{Ld_2} r_k b_k e_{j_k}^\top}_{*} 
= \sup_{\vert\vert\vert M \vert\vert\vert \leq 1} \left| \sum_{k=1}^{Ld_2} \langle r_k b_k, M e_{j_k} \rangle \right| 
= \underbrace{ \sup_{\vert\vert\vert M \vert\vert\vert \leq 1} \left | \sum_{k=1}^{Ld_2} \langle b_k, M e_{j_k} \rangle  \right |}_{\text{($\S$)}},
\end{aligned}
\end{equation*}
where the second equation holds in the sense of distribution.

A tail bound on ($\S$) is derived by the following lemma, which is a direct consequence of the moment version of Dudley's inequality (e.g., p. 263 in \citep{foucart2013mathematical}) and a version of Markov's inequality (e.g., Proposition~7.11 in \citep{foucart2013mathematical}).

\begin{lemma}
\label{lemma:gaussianmax}
Let $\xi \sim \mathcal{N}(0,I_n)$, $\Delta \subset \mathbb{R}^n$, and $0 < \zeta < e^{1/2}$. 
Then there exists constant $c$ such that
\[ \sup_{f \in \Delta}  \left | f^* \xi \right | \leq c \sqrt{\log (\zeta^{-1}) } \int_{0}^{\infty} \sqrt{\log N(\Delta, \eta B_2)} d\eta \] with probability $1 - \zeta$.
\end{lemma}

Since the set $\{M: \tnorm{M} \leq \alpha\}$ is symmetric, we can omit the absolute value in the objective function in ($\S$). 
Let $f_M := [\mathbf{1}_{L,1} \otimes (M e_1) ; \dots ; \mathbf{1}_{L,1} \otimes (M e_{d_2})] \in \mathbb{R}^{Ld_2}$, where $\mathbf{1}_{L,1}$ denotes the column vector of length $L$ with all entries set to 1.
Then ($\S$) is written as the maximum of $f^* \xi$ with $\xi \sim \mathcal{N}(0,I_n)$ over the set $\Delta = \{ f_M : \tnorm{M} \leq 1 \}$.
Since 
\begin{align*}
\norm{f_M - f_{M'}}_2 
= \sqrt{L} \, \norm{M - M'}_\mathrm{F}
\leq \norm{M - M'}_{1\to 2} \sqrt{Ld_2}, 
\end{align*}
it follows that
\[
N(\Delta, \eta B_2) 
\leq N\left( B_{\vert\vert\vert\cdot\vert\vert\vert}, \frac{\eta}{\sqrt{Ld_2}}B_{1\to 2} \right)
\leq N\left( \sqrt{r} B_\mathrm{mixed}, \frac{\eta}{\sqrt{Ld_2}}B_{1\to 2} \right).
\]
Recall that the mixed-norm is equivalent to the projective tensor norm up to a constant. 
Therefore Proposition~\ref{prop:e21_pi2e_ellinfell2} provides an upper bound on $N(\Delta,\eta B_2)$. 
The assertion is obtained by applying this to Lemma~\ref{lemma:gaussianmax}.
\end{proof}

Finally, plugging in the tail estimates of $\theta$ and $\Gamma$ respectively by Lemmas~\ref{lemma:rip_dc} and \ref{lemma:gc_dc} to Proposition~\ref{prop:deterministic_errbnd} provides Theorem~\ref{thm:eb_ds}.
We choose $\zeta = (d_1+d_2)^{-1}$ in applying those lemmas.

\subsection{Completion of bounded and approximately low-rank matrices}
\label{sec:MC}

Similarly, Proposition~\ref{prop:deterministic_errbnd} can also be used to obtain the error bound for max-norm-constrained matrix completion. 
The error bound matches the best-known result \citep{cai2016matrix} (up to a logarithmic factor).
Indeed, Proposition~\ref{prop:deterministic_errbnd} is an abstraction of the proof strategy by \citet{cai2016matrix} to general tensor-norm-based LASSO. 
\citet{cai2016matrix} have already provided tail estimates on $\theta$ and $\Gamma$. 
Here we provide an alternative tail bound on $\theta$ through Maurey's empirical method. 

By applying the entropy estimate in Proposition~\ref{prop:e21_pi2e_ellinfellinf} to a version of the Rudelson-Vershynin lemma \cite[Proposition~2.6]{junge2020generalized}, which generalizes upon previous works \citep{rudelson2008sparse,rauhut2010compressive,dirksen2015tail}, we obtain the following concentration inequality on the quadratic form with random entry-wise sampling operator.

\begin{lemma}
\label{lemma:rip_mc}
Let $(i_k,j_k)$ for $k=1,\dots,n$ be independent copies of a uniform random variable on $\{1,\dots,d_1\} \times \{1,\dots,d_2\}$. Let $\tnorm{\cdot}$ be defined as in \eqref{eq:tnorm_max}. 
Then there exists an absolute constant $C$ such that
\begin{equation}
\label{eq:lrf_rip}
\begin{aligned}
& \sup_{\vert\vert\vert M \vert\vert\vert \leq 1}
\left| \frac{d_1 d_2}{n} \sum_{k=1}^n |\langle e_{i_k} \otimes e_{j_k}, M \rangle|^2 - \|M\|_\mathrm{F}^2 \right| \\
& \leq
C d_1 d_2 \left( \frac{ r(d_1+d_2)\ln^3(d_1+d_2) \vee \ln^2(\zeta^{-1})}{n} \vee \sqrt{\frac{ r(d_1+d_2)\ln^3(d_1+d_2) \vee \ln^2(\zeta^{-1})}{n}} \right)
\end{aligned}
\end{equation}
holds with probability at least $1-\zeta$.
\end{lemma}

\begin{proof}
The assertion is obtained as a consequence of \cite[Proposition~2.6]{junge2020generalized}. 
Let $X$ be a Banach space of $d_1$-by-$d_2$ matrices equipped with the norm defined by $\|M\|_X = \sqrt{d_1 d_2} \tnorm{M}$. 
Then it immediately follows that $\|M\|_\mathrm{F} \leq \|M\|_X$. 
We consider a generalized sparsity model given by $\{M: \|M\|_X \leq \sqrt{s} \|M\|_\mathrm{F} \}$. 
Let $v_k: X \to \mathbb{R}$ denote the linear operator defined by
\[
v_k(M) = \langle \sqrt{d_1d_2} \, e_{i_k} \otimes e_{j_k}, M \rangle\,
\]
for $k = 1,\dots,n$. 
Since the random indices are uniformly distributed, we have the isotropy, i.e. $\mathbb{E} v_k^* v_k = \mathrm{Id}$. 
Let $v: X \to \ell_\infty^n$ denote the composite operator such that $v(M) = [v_1(M), \dots, v_n(M)]^\top$. 
Then \cite[Proposition~2.6]{junge2020generalized}\footnote{\cite[Proposition~2.6]{junge2020generalized} has an extra constraint $\norm{M}_\mathrm{F} = 1$ in the supremum in the left-hand side. This leads to a multiplicative deviation, i.e. the upper bound is proportional to $\norm{M}_\mathrm{F}$. However, the proof remains valid when this unit-norm constraint is dropped.} implies that there exists an absolute constant $C_1$ with which 
\begin{equation}
\label{eq:lrf_rip2}
\begin{aligned}
& \sup_{
\norm{M}_X \leq \sqrt{s} 
} 
\left| \frac{1}{n} \sum_{k=1}^n |v_k(M)|^2 - \|M\|_\mathrm{F}^2 \right| \leq C_1 \left( \sqrt{\frac{s \varrho}{n}} \vee \frac{s \varrho}{n}\right) 
\end{aligned}
\end{equation}
holds with probability at least $1-\zeta$, where
\[
\varrho := \sup_{k \in \mathbb{N}} \left[(\mathbb{E} \mathcal{E}_{2,1}(v)^{2k})^{1/k} + \sqrt{\ln(\zeta^{-1}}) (\mathbb{E} \|v\|^{2k})^{1/k}\right].
\]
Therefore, to get an upper bound in \eqref{eq:lrf_rip}, it suffices to calculate the moments of $\mathcal{E}_{2,1}(v)$ and $\|v\|$. 

First we obtain an upper bound on $\mathcal{E}_{2,1}(v)$ by Proposition~\ref{prop:e21_pi2e_ellinfellinf} as follows.
\begin{center}
\begin{tikzcd}[row sep=large, column sep=large, ar symbol/.style = {draw=none,"#1" description,sloped}, isomorphic/.style = {ar symbol={\cong}}]
X \arrow{dr}{\mathrm{Id}} \arrow{rr}{v} & & \ell_\infty^n \\
& \ell_\infty^{d_2} \mathbin{\widecheck{\otimes}} \ell_\infty^{d_1} \arrow{ur}{v}
\end{tikzcd}
\end{center}
By the above commutative diagram, we have
\begin{align*}
\mathcal{E}_{2,1}(v: X \to \ell_\infty^n) 
& \leq \|v: \ell_\infty^{d_2} \mathbin{\widecheck{\otimes}} \ell_\infty^{d_1} \to \ell_\infty^n\| \cdot \mathcal{E}_{2,1}(\mathrm{Id}: X \to \ell_\infty^{d_2} \mathbin{\widecheck{\otimes}} \ell_\infty^{d_1}) \\
& \leq \sqrt{d_1d_2} \cdot \sqrt{\frac{r}{d_1d_2}} \cdot  \mathcal{E}_{2,1}(\mathrm{Id}: \ell_\infty^{d_2} \mathbin{\widehat{\otimes}} \ell_\infty^{d_1} \to \ell_\infty^{d_2} \mathbin{\widecheck{\otimes}} \ell_\infty^{d_1}) \\
& \leq C_2 \sqrt{r(d_1+d_2) \ln^3(d_1+d_2)},
\end{align*}
where the last inequality follows from Proposition~\ref{prop:e21_pi2e_ellinfellinf}.

Moreover, the operator norm of $v$ satisfies
\[
\|v: X \to \ell_\infty^n\| 
\leq \frac{1}{\sqrt{d_1d_2}} \|v : \ell_\infty^{d_2} \mathbin{\widecheck{\otimes}} \ell_\infty^{d_1} \to \ell_\infty^n \| = 1.
\]
Note that the above upper bounds on $\mathcal{E}_{2,1}(v)$ and $\|v\|$ hold with probability $1$. 
Therefore the corresponding moment terms are upper-bounded in the same way. 
Then the assertion follows by plugging in these upper estimates to \eqref{eq:lrf_rip2} with $s = d_1 d_2$.
\end{proof} 

This concentration inequality provides an alternative tail bound on $\theta$. Lemma~\ref{lemma:rip_mc} provides 
\begin{align}
\sup_{\vert\vert\vert M \vert\vert\vert \leq 1} \left| \frac{d_1 d_2}{n} \sum_{k=1}^n |\langle e_{i_k} \otimes e_{j_k}, M \rangle|^2 - \|M\|_\mathrm{F}^2 \right| 
\leq
C \sqrt{\frac{d_1 d_2 r(m+n)\ln^3(m+n)}{n}}
\label{eq:cai_comp_ours} 
\end{align}
with high probability. 
Compared to the analogous result by \citet{cai2016matrix}, which is summarized as follows: 
\begin{align}
\left| \frac{d_1 d_2}{n} \sum_{k=1}^n |\langle e_{i_k} \otimes e_{j_k}, M \rangle|^2 - \|M\|_\mathrm{F}^2 \right| 
& \leq \frac{1}{2} \, \|M\|_{\mathrm{F}}^2+
C \sqrt{\frac{r (d_1+d_2)}{n}} \label{eq:cai_comp_cai}
\end{align} 
holds for all $M$ with $\tnorm{M} \leq 1$ with high probability. 
Our upper bound in \eqref{eq:cai_comp_ours} has an extra logarithmic factor compared to \eqref{eq:cai_comp_cai}.
Therefore, unlike the result in \citep{cai2016matrix}, we do not achieve an optimal rate only up to an absolute constant. 
However, our proof by Maurey's empirical method \citep{carl1985inequalities} might apply to a broader class of tensor norms beyond the max-norm by leveraging sophisticated embedding theorems (e.g. \citep{kashin1977diameters,figiel1977dimension,szarek1978kashins,szarek1980nearly,johnson1982embeddingl,pisier1983dimension,talagrand1998selecting,litvak2005euclidean,guedon2007subspaces,friedland2011random}). 

\section{Information-theoretic lower bound}
\label{sec:minimax}

\citet{cai2016matrix} have established a minimax lower bound for matrix completion where the unknown matrix is approximately low-rank and bounded.  
We present the proof of Theorem~\ref{thm:minimax_ds} by adapting their strategy to our setting of decentralized sketching. 
Let us first recall the notion of packing (e.g. see \citep[Definition~4.2.4]{vershynin2018high}): For any set $\mathcal{S}$, a subset $\mathcal{P} = \{ x_1, x_2, \ \cdots \ x_n \} \subset \mathcal{S}$ is called $\epsilon$-packing of $\mathcal{S}$ if $d(x_i, x_j) > \epsilon$ for all $i \neq j \in [n]$. 
The parameter $\epsilon$ denotes the packing density. 
Further, the packing number of $\mathcal{S}$ is defined as 
\begin{equation}
    \label{eq:def_packing_num}
    M(\mathcal{S},d,\ \epsilon) = \max\{n \in \mathbb{N} : \text{$\exists$ $\epsilon$-packing set of $\mathcal{S}$ of size $n$} \}.
\end{equation} We now proceed with the proof.

Let $\kappa(\alpha, R) = \{M : \norm{M}_{1\to 2} \leq \alpha, ~ \norm{M}_\mathrm{mixed} \leq R\}$ with $R = \sqrt{r} \alpha$. 
The first step is to show that there exists a packing set of $\kappa(\alpha, R)$ of a desirable size and packing density, 
which is followed by a multiway hypothesis testing argument and Fano's inequality. 
These steps will establish the minimax lower bound. 

The following lemma, obtained by adapting the statement and proof of \cite[Lemma 3.1]{cai2016matrix} to our setting, provides a random construction of a packing set of $\kappa(\alpha,R)$.
\begin{lemma} 
Let $r = R^2/\alpha^2$ and $\gamma \leq 1$ satisfy $r \leq \gamma^2 (d_1 \wedge d_2)$ is an integer. Then there exists a subset $\mathcal{M} \subset \kappa(\alpha, R)$ with cardinality 
\[ 
|\mathcal{M}| = \left\lceil \exp \left\{ \frac{r(d_1 \vee d_2)}{16 \gamma^2} \right\} \right\rceil 
\] 
with the following properties:
\begin{enumerate}
    \item Every $M \in \mathcal{M}$ satisfies that $\mathrm{rank}(M) \leq r/\gamma^2$ and $M_{kl} \in \{ \pm \gamma \alpha/\sqrt{d_1} \}$ for all $k \in [d_1]$ and $l \in [d_2]$, thereby
    \[ 
    \norm{M}_{1 \rightarrow 2} = \gamma \alpha, \quad \text{and} \quad
    \norm{M}_\mathrm{F}^2 = \gamma^2 \alpha^2 d_2.
    \]

    \item Any two distinct $M^i, M^j \in \mathcal{M}$ satisfy
    \[ 
    \norm{M^i - M^j}_\mathrm{F}^2 
    \geq \frac{\gamma^2 \alpha^2d_2}{2}.
    \]
\end{enumerate}
\label{lm:pack_set}
\end{lemma}

\begin{proof}
The idea is to show the existence of $\mathcal{M}$ by  leveraging the empirical method. 
Without loss of generality, we may assume that $d_2 \geq d_1$. (Otherwise we only need to flip $d_1$ and $d_2$ in the first assumption.) 
Let $N = \exp(r d_2/16 \gamma^2)$, $B = r/\gamma^2$ and for each $i = 1, \ \cdots, \ N$, we draw a random matrix $M^i$ as follows: The matrix $M^i$ consists of i.i.d. blocks of dimensions $B \times d_2$, stacked up from top to bottom, with entries of the first block being i.i.d. symmetric random variables, taking values in $\pm \alpha \gamma/\sqrt{d_1}$ such that 
\[ 
M^i_{kl} = M^{i}_{k'l}, \quad \forall l, \, \forall k, k' : k' \equiv k \, (\mathrm{mod} \, B). 
\]
It can be verified that all the matrices $M^1,\dots,M^N$ drawn in such a manner satisfy the first property above. We can then define the packing set $\mathcal{M}$ as the set $\{ M_1,M_2, \cdots, M_N\}$. It remains to show that the second property is also satisfied.

For any $M^i \neq M^j$, we have 
\begin{align*}
    \|M^i - M^j\|_\mathrm{F}^2 
    = \sum_{k,l} (M^i_{kl} - M^j_{kl})^2
    \geq \left\lfloor \frac{d_1}{B} \right\rfloor \sum_{k=1}^{B} \sum_{l=1}^{d_2} (M^i_{kl} - M^j_{kl})^2 
    = \frac{4\alpha^2 \gamma^2}{d_1} \left\lfloor \frac{d_1}{B} \right\rfloor \sum_{k=1}^{B} \sum_{l=1} ^{d_2} \delta_{kl},
\end{align*} 
where $(\delta_{kl})$ denotes an array of i.i.d. Bernoulli random variables with mean $ 1/2$. 
Then Hoeffding's inequality implies  
\begin{equation*}
    \mathbb{P} \left ( \sum_{k=1}^B \sum_{l=1}^{d_2} \delta_{kl} \geq \frac{Bd_2}{4} \right ) 
    \leq e^{- B d_2 / 8}.
\end{equation*} 
By using the union bound argument over all $\binom{N}{2}$ possible pairs, we obtain that 
\begin{equation}
    \label{eq:separation}
    \min_{i \neq j} \|M^i - M^j\|_\mathrm{F}^2
    > \alpha^2 \gamma^2 \left\lfloor \frac{d_1}{B} \right\rfloor \frac{Bd_2}{d_1} 
    \geq \frac{\alpha^2 \gamma^2  d_2}{2}
\end{equation} 
holds with probability at least $1 - \binom{N}{2} \exp \left( -Bd_2/8 \right ) \geq 1/2$. 
In other words, the second property is satisfied with nonzero probability, thereby, there exists such an instance satisfying \eqref{eq:separation}. 
This concludes the proof. 
\end{proof}

\begin{lemma}
Let $\mathcal{M} \subset \kappa(\alpha,R)$ be $\delta$-separated, i.e. every distinct pair in $\mathcal{M}$ is at least separated by $\delta$ in the Frobenius norm. 
Let $\widetilde{M} = \mathrm{arg min}_{M \in \mathcal{M}} \|M - \Phi\|_\mathrm{F}$ and $M^*$ be uniformly distributed over $\mathcal{M}$. 
Then 
\begin{equation*}
    \inf_{\Phi} \sup_{M \in \kappa(\alpha,R)} \E \|\Phi - M\|_\mathrm{F}^2 
    \geq 
    \frac{\delta^2}{4} \min_{\widetilde{M} \in \mathcal{M}} \mathbb{P}( \widetilde{M} \neq M^* ). 
\end{equation*} 
\end{lemma}

\begin{proof}
Suppose that there exists $M^j \in \mathcal{M}$ such that $M^j \neq \widetilde{M}$. 
By the optimality of $\widetilde{M}$ and the triangle inequality, we have
\begin{align*}
    \|\Phi - M^j\|_\mathrm{F} 
    \geq \|M^j - \widetilde{M}\|_\mathrm{F} - \|\Phi - \widetilde{M}\|_\mathrm{F}
    \geq \|M^j - \widetilde{M}\|_\mathrm{F} - \|\Phi - M^j\|_\mathrm{F},
\end{align*}
which implies
\[
    \|\Phi - M^j\|_\mathrm{F} 
    \geq \frac{\|M^j - \widetilde{M}\|_\mathrm{F}}{2}.
\]
Thus, since $M^j, \widetilde{M} \in \mathcal{M}$ satisfy $\|M^j - \widetilde{M}\|_\mathrm{F} \geq \delta$, we obtain
\[
    \|\Phi - M^j\|_\mathrm{F}^2 \geq \frac{\delta^2}{4}.
\]
Hence we deduce that
\begin{equation}
    M^j \neq \widetilde{M} 
    \implies 
    \|\Phi - M^j\|_\mathrm{F}^2
    \geq \frac{\delta^2}{4}.
    \label{eq:MtildetoMhat}
\end{equation}

Finally, we have 
\begin{align*}
    \inf_{\Phi} \sup_{M \in \kappa(\alpha,R)} \E \|\Phi - M\|_\mathrm{F}^2 
    & \geq \inf_{\Phi} \max_{M^j \in \mathcal{M}} \E \|\Phi - M^j\|_\mathrm{F}^2 \\
    & \overset{\mathrm{(a)}}{\geq} \inf_{\Phi} \max_{M^j \in \mathcal{M}} \frac{\delta^2}{4} \mathbb{P}\left( \|\Phi - M^j\|_\mathrm{F}^2 \geq \frac{\delta^2}{4} \right) \\
    & \overset{\mathrm{(b)}}{\geq} \frac{\delta^2}{4} \min_{\widetilde{M} \in \mathcal{M}} \max_{M^j \in \mathcal{M}} \mathbb{P}( \widetilde{M} \neq M^j ) \\
    & \overset{\mathrm{(c)}}{\geq} \frac{\delta^2}{4} \min_{\widetilde{M} \in \mathcal{M}} \mathbb{P}( \widetilde{M} \neq M^* ),
\end{align*} 
where (a) holds by Markov's inequality; (b) follows from \eqref{eq:MtildetoMhat}; and (c) holds since the worst-case error probability is larger than the error probability with respect to the uniformly distributed random matrix $M^*$.
This completes the proof.
\end{proof}

We now proceed to provide a lower bound on $\min_{\widetilde{M} \in \mathcal{M}} \mathbb{P}( \widetilde{M} \neq M^*)$ by using the following lemma, which is a consequence of Fano's inequality (Theorem 2.10.1,  \citep{cover1991elements}).

\begin{lemma}[{\citep[Eq. (6.17)]{cai2016matrix}}]
Let $P(y|M,(A_k)_{k=1}^n)$ denote the conditional probability density of $y$ in \eqref{eq:meas_mdl} given $M$ and $(A_k)_{k=1}^n$. Let $D_{\mathrm{KL}}(M \,\|\, M')$ denote the Kullback-Leiber divergence between $p(y|M,(A_k)_{k=1}^n)$ and $p(y|M',(A_k)_{k=1}^n)$. 
Then
\begin{align}
\label{eq:Fanos}
     \mathbb{P} (\widetilde{M} \neq M^*) & \geq 1 - \frac{ {\binom{|\mathcal{M}|}{2}} ^{-1} \sum_{j \neq j'} \mathbb{E} \, D_{\mathrm{KL}} (M^j \,\|\, M^{j'})+ \log 2}{\log | \mathcal{M}|},
\end{align} 
where $\mathcal{M}$ is the packing set derived in Lemma~\ref{lm:pack_set}.
\end{lemma}

\noindent Moreover, with Gaussian noise, $KL(M^j||M^{j'})$ is simplified as follows.
\begin{lemma} 
Suppose that $\eta_1,\dots,\eta_n$ are i.i.d. $\mathcal{N}(0,\sigma^2)$. 
Then
\begin{equation*}
     D_{\mathrm{KL}}(M^j \,\|\, M^{j'}) = \frac{1}{2 \sigma^2} \sum_{k=1}^n |\langle A_k, M^j - M^{j'} \rangle|^2.
\end{equation*} 
Furthermore, if $\mathbb{E} A_k \langle A_k, M \rangle = M$ for all $M$, then
\begin{equation*}
    \mathbb{E} \, D_\mathrm{KL}(M^j \,\|\, M^{j'}) = \frac{1}{2\sigma^2} \| M^j - M^{j'} \|_\mathrm{F}^2.
\end{equation*}
\end{lemma}

\begin{proof}
Since the conditional density function of $y$ is written as
\begin{align*}
    p(y|M,(A_k)_{k=1}^n) 
    = \prod_{k=1}^{n} \frac{1}{\sqrt{2 \pi \sigma^2}} \exp\left( - \frac{1}{2\sigma^2} \left ( y_k - \langle A_k, M \rangle \right )^2 \right),
\end{align*}
it follows that
\begin{align*}
    & D_\mathrm{KL}(M^j \,\|\, M^{j'}) \\
    & = \int p(y|M^j,(A_k)_{k=1}^n) \log \left ( \frac{p(y|M^j,(A_k)_{k=1}^n)}{p(y|M^{j'},(A_k)_{k=1}^n)} \right ) dy \\
    & = \mathbb{E} \left( \left. \sum_{k=1}^n \frac{(y_k - \langle A_k, M^{j'} \rangle)^2 - (\sqrt{L} y_k - \langle A_k, M^j \rangle)^2}{ 2\sigma^2} \right| M^j, (A_k)_{k=1}^n \right) \\
    & = \frac{1}{2\sigma^2} \sum_{k=1}^n |\langle A_k, M^{j'}-M^j \rangle|^2.
\end{align*}
The second part follows immediately from the isotropy assumption. 
\end{proof}

Lemma~\ref{lm:pack_set} implies that $\|M^j - M^{j'}\|_\mathrm{F}^2 \leq 4 \alpha^2 \gamma^2 d_2$ for all distinct $M^j$ and $M^{j'}$ in $\mathcal{M}$. 
Furthermore, we also have $\log |\mathcal{M}| \geq r(d_1 \vee d_2)/(16\gamma^2)$. With this, we want to identify the conditions under which the above probability is non-zero. Let us suppose that  $\gamma^4 \leq \sigma^2 r (d_1 \vee d_2) / (128 \alpha^2 d_2)$ and $r(d_1 \vee d_2) \geq 48$. Then, using the above estimate of the KL divergence in from \eqref{eq:Fanos} provides
\begin{align*}
    \mathbb{P} (\widetilde{M} \neq M^*)
    \geq 1 - \frac{16 \gamma^2}{r (d_1 \vee d_2)} \left(\frac{2 \alpha^2 \gamma^2 d_2}{\sigma^2} + \log 2\right)
    \geq \frac{1}{2}.
\end{align*} We already have that $r(d_1 \vee d_2) \geq 48$ from the assumption stated in Theorem \ref{thm:minimax_ds}. We now show how the condition on $\gamma^4$ can be satisfied. If $\sigma^2 r (d_1 \vee d_2) \geq 128 \alpha^2 d_2$, then we can choose $\gamma^2 = 1$. In this case, we obtain 
\begin{align*}
     \inf_{\Phi} \sup_{M \in \kappa(\alpha,R)} \E \|\Phi - M\|_\mathrm{F}^2 
     \geq \frac{\delta^2}{4} \cdot \frac{1}{2}
     \geq \frac{\alpha^2 d_2}{16},
\end{align*} 
since $\delta = \alpha\gamma\sqrt{d_2/2}$. 
Hence 
\begin{align*}
     \inf_{\Phi} \sup_{M \in \kappa(\alpha,R)} \E \frac{1}{d_2} \|\Phi - M\|_\mathrm{F}^2
     \geq \frac{\delta^2}{4} \cdot \frac{1}{2}
     \geq \frac{\alpha^2}{16}.
\end{align*} 
Otherwise, we choose $\gamma^2 = \sqrt{\sigma^2 r (d_1 \vee d_2)/(128 \alpha^2 d_2)}$ and obtain
\begin{align*}
     \inf_{\Phi} \sup_{M \in \kappa(\alpha,R)} \E \|\Phi - M\|_\mathrm{F}^2 & \geq \frac{\delta^2}{4} \cdot \frac{1}{2} \\
     & \geq \alpha^2 \sigma  \sqrt{\frac{r (d_1 \vee d_2)}{128}} \cdot \frac{1}{\alpha \sqrt{d_2}} \cdot \frac{d_2}{2} \cdot \frac{1}{4}\\
     & \geq \alpha \sigma \sqrt{\frac{r(d_1 \vee d_2)}{128}} \cdot \frac{\sqrt{d_2}}{2}
\end{align*} 
since $\delta = \alpha\gamma\sqrt{d_2/2}$. Note that this choice of $\gamma^2$ obeys the conditions in Lemma \ref{lm:pack_set}. 
Hence 
\begin{align*}
     \inf_{\Phi} \sup_{M \in \kappa(\alpha,R)} \E \frac{1}{d_2} \|\Phi - M\|_\mathrm{F}^2 & \geq \frac{\alpha^2}{16} \cdot \frac{\sigma\sqrt{L}}{\alpha} \sqrt{\frac{r(d_1 \vee d_2)}{L d_2}}
\end{align*} 

Finally, combining the above results, we obtain
\begin{equation*}
    {\inf_{\Phi}} {\sup_{M\in \kappa(\alpha, R)}} \frac{1}{d_2} \E \|\Phi - M\|_\mathrm{F}^2 \geq \min \left ( \frac{\alpha^2}{16},  \frac{\alpha^2}{16} \cdot \frac{\sigma\sqrt{L}}{\alpha} \sqrt{\frac{r(d_1 + d_2)}{L d_2}} \right ).
\end{equation*}

To simplify further, we can consider the case where the number of measurements $Ld_2$ is greater than the degrees of freedom: 
\begin{align*}
    Ld_2 & \geq r(d_1+d_2) = \frac{R^2(d_1+d_2)}{\alpha^2}.
\end{align*} 
Then, we have that $\alpha^2 \geq \frac{R^2(d_1+d_2)}{Ld_2}$. 
Under this condition, the above minimax lower bound reduces to
\begin{align*}
    {\inf_{\Phi}} {\sup_{M\in \kappa(\alpha, R)}} \frac{1}{d_2} \E \|\Phi - M\|_\mathrm{F}^2 
    & \geq \frac{\alpha^2}{16} \wedge \frac{\alpha^2}{16} \cdot \frac{\sigma\sqrt{L}}{\alpha} \sqrt{\frac{r(d_1 + d_2)}{L d_2}} \\
    & \geq \frac{\alpha R}{16} \sqrt{\frac{d_1+d_2}{Ld_2}} \wedge  \frac{\alpha R}{16} \cdot \frac{\sigma \sqrt{L}}{\alpha} \sqrt{\frac{d_1+d_2}{Ld_2}} \\
    & \geq  \frac{R}{16}\sqrt{\frac{d_1+d_2}{Ld_2}} \left ( \alpha \wedge \sigma \sqrt{L} \right ).
\end{align*}
Therefore, we have the following minimax lower bound:
\begin{equation*}
    {\inf_{\Phi}} \sup_{M_0 \in K_\mathrm{mixed}} \frac{\E \|\Phi - M\|_\mathrm{F}^2 }{d_1 d_2} \geq \frac{\alpha^2}{16 d_1} \sqrt{\frac{r(d_1 + d_2)}{L d_2}} \left( 1 \wedge \frac{\sigma\sqrt{L}}{\alpha}  \right).
\end{equation*}

\section{Discussion}

We presented a tensor-norm-constrained LASSO estimator for low-rank recovery and its statistical analysis. 
When the observations are obtained by a local measurement operator, we proposed a principled design of a tensor-norm regularizer adapting to the local structure of measurements. 
Such a tensor norm that generalizes the max-norm and mixed-norm has provided a near optimal error bound over decentralized subspace sketching and matrix completion. 
Importantly, the relaxed low-rank model allows inexact modeling of data. 
Furthermore, the error bound applies uniformly to all approximately low-rank matrices satisfying a less stringent condition than the conventional incoherence condition. 
We will provide more applications of the presented framework in companion papers including subspace blind deconvolution and recovery of jointly low-rank matrices.

There are several interesting directions to which the presented results may extend. 
Much less is known about tensor products of three or more Banach spaces compared to the case of two Banach spaces. 
Recently, the geometry of selected triple tensor products has been studied \citep{giladi2017geometry}. 
The extension along this direction will advance understanding on the regression with multilinear data. 
With applications in quantum tomography, it will be also fruitful to study the extension to tensor products of operator spaces similarly to the diamond norm approach in \citep{kliesch2016improving}.

\section*{Funding}
This work was supported in part by NSF CCF-1718771, DMS-1800872, and CAREER Award CCF-1943201. 

\section*{Acknowledgment}
K.L. thanks Felix Krahmer and Dominik St\"oger for helpful discussions. 
The authors thank the anonymous reviewers for their constructive comments, which helped significantly improve the presentation of this paper.

\appendix

\section{Proof of Proposition~\ref{prop:deterministic_errbnd}}
\label{sec:proof:prop:deterministic_errbnd}

First, note that since both $M_0$ and $\Phi$ are feasible, we have 
\[
	\tnorm{\Phi-M_0}\leq 2\alpha,
\]
and since $\Phi$ is optimal, 
\begin{equation*}
	\sum_{k=1}^n \left( y_k - \mathrm{tr}(A_k^\top \Phi)\right )^2 \leq 
	\sum_{k=1}^n \left( y_k - \mathrm{tr}(A_k^\top M_0)\right )^2.
\end{equation*} 
Then under the assumption \eqref{eq:qub}, and using the fact that $y_k=\mathrm{tr}(A_k^\top M_0) + \eta_k$, we have
\begin{align*}
	\norm{\Phi - M_0}_\mathrm{F}^2  
	& \leq 4 \alpha^2 \theta + \frac{1}{n} \sum_{k=1}^n\mathrm{tr}(A_k^\top(\Phi - M_0))^2 \\
	&= 4 \alpha^2 \theta + \frac{1}{n} \sum_{k=1}^n \left[ (y_k-\mathrm{tr}(A_k^\top\Phi))^2 + 2\eta_k(\mathrm{tr}(A_k^\top\Phi)-y_k) + \eta_k^2 \right] \\
	&\leq 4 \alpha^2 \theta + \frac{1}{n} \sum_{k=1}^n \left[ (y_k-\mathrm{tr}(A_k^\top M_0))^2 + 2\eta_k(\mathrm{tr}(A_k^\top\Phi)-y_k) + \eta_k^2 \right] \\
	&= 4 \alpha^2 \theta + \frac{2}{n} \sum_{k=1}^n \eta_k\mathrm{tr}(A_k^\top(\Phi - M_0)).
\end{align*}
To upper bound the right-hand side, we take the supremum over all feasible $\tilde{M}$,
\[
	\sup_{\vert\vert\vert \tilde{M} \vert\vert\vert \leq \alpha} \sum_{k=1}^n \eta_k\mathrm{tr}\left(A_k^\top(\tilde{M} - M_0)\right) 
	= \sup_{\vert\vert\vert M \vert\vert\vert \leq 2 \alpha} \sum_{k=1}^n \langle \eta_k A_k, M \rangle 
	= 2\alpha \tnorm{\sum_{k=1}^n \eta_k A_k }_*.
\]
The quantity on the right is a Gaussian empirical process for which \cite[Theorem~4.7]{pisier1999volume} provides the concentration bound
\begin{align*}
	\tnorm{\sum_{k=1}^n \eta_{k} A_{k}}_{*} &
	\leq \sigma\E \tnorm{\sum_{k=1}^n g_{k} A_{k} }_{*} + \sigma\pi \sqrt{\frac{\log(2\zeta^{-1})}{2} \sup_{\vert\vert\vert M \vert\vert\vert \leq 1} \sum_{k=1}^n \mathrm{tr}(A_{k}^\top M)^2 },
\end{align*}
which holds with probability $1-\zeta$, where the $(g_k)_{k=1}^n$ are i.i.d. Gaussian with zero mean and unit variance. 
Thus
\begin{align*}
	\| \Phi - M_0\|_\mathrm{F}^2 
	&\leq 4 \alpha^2 \theta + \frac{4\alpha\sigma\Gamma}{n} + \frac{4\alpha\sigma\pi}{n} \sqrt{\frac{\log(2\zeta^{-1})}{2} \sup_{ \vert\vert\vert M \vert\vert\vert \leq 1} \sum_{k=1}^n \mathrm{tr}(A_{k}^\top M)^2 } \\
	&\leq 4 \alpha^2 \theta + \frac{4\alpha\sigma\alpha\Gamma}{n} + \frac{4\alpha\sigma\alpha\pi}{n} \sqrt{\frac{n\log(2\zeta^{-1})}{2}
		\left(\theta + \sup_{  \vert\vert\vert M \vert\vert\vert\leq 1}\|M\|_\mathrm{F}^2\right)} \\
	&\leq 4 \alpha^2 \theta + \frac{4\alpha\sigma\Gamma}{n} + 2\pi\alpha\sigma \sqrt{\frac{2\log(2\zeta^{-1})\left(\theta + R^2\right)}{n}}.
\end{align*}

\section{Proof of Lemma~\ref{lemma:pi2a_op}}
\label{sec:proof:lemma:pi2a_op}
Let $E = \mathrm{ker}(T^*)^\perp$. Then it follows that $T^* = T^* P_E$, where $P_E$ denotes the orthogonal projection onto $E$. Furthermore, since $\mathrm{rank}(T^*) \leq r$, we have $\mathrm{dim}(E) \leq r$. 
By \cite[Proposition~3.1]{jameson1987summing}, the 2-summing norm of $T^*$ is upper-bounded by
\[
\pi_2(T^*) \leq \|T^*\| \pi_2(P_E).
\]
Since the $\ell_2$ norm is unitarily-invariant, there exists a diagonal operator $\Lambda$ such that $\pi_2(P_E) = \pi_2(\Lambda)$. 
Then, since $\mathrm{dim}(E) \leq r$, it follows by \cite[Proposition~3.5]{jameson1987summing} that
\[
\pi_2(P_E) \leq \sqrt{r}.
\]
Finally it trivially holds that $\|T^*\| = \|T\|$. 
Combining the results provides the upper bound. 
Next, the lower bound is obtained from the fact that the 2-summing norm is always lower-bounded by the operator norm (see \cite[Proposition~3.1]{jameson1987summing}). 
This completes the proof.

\section{Proof of Lemma~\ref{lemma:p2ANDp}}
\label{sec:proof:lemma:p2ANDp}
We use the following result by Pietsch to prove Lemma~\ref{lemma:p2ANDp}.

\begin{lemma}[{\cite[Theorem~5.8]{jameson1987summing}}]
\label{lemma:pietsch}
Let $X,Y$ be normed linear spaces ($Y$ complete) and $T \in L(X,Y)$ be a 2-summing operator. Then there exists a Hilbert space $H$ and operators $T_1 \in L(X,H)$ and $T_2 \in L(H,Y)$ such that $T = T_2 T_1$, $\pi_2(T_1) = \pi_2(T)$, and $\nu_1(T_2) = 1$.
\end{lemma}

By Lemma~\ref{lemma:pietsch}, there exist $d \in \mathbb{N}$, $T_1^* \in L(Y^*,\ell_2^d)$, and $T_2^* \in L(\ell_2^d,X)$ such that $T^* = T_2^* T_1^*$, $\pi_2(T_1^*) = \pi_2(T^*)$ and $\|T_2^*\| = 1$. Furthermore, by \cite[Proposition~3.1]{jameson1987summing}, we have
\[
\pi_2(T_2^* T_1^*) \leq \pi_2(T_1^*) \|T_2^*\|
\]
for any $T_1^* \in L(Y^*,\ell_2^d)$ and $T_2^* \in L(\ell_2^d,X)$. Therefore, the 2-summing norm of $T^*$ is written as in \eqref{eq:pi2Tstar1}.

\section{Proof of Lemma~\ref{lemma:pi2a_nu1}}
\label{sec:proof:lemma:pi2a_nu1}
We first compute the nuclear-$1$ norm of the adjoint $T^* \in Y \otimes \ell_\infty^n$. 
By the trace duality, the projective tensor norm of $T^*$ is written as
\begin{equation}
\label{eq:pi_dual}
\nu_1(T^*) = \sup \{ \mathrm{tr}(S^*T^*) : S^* \in \ell_1^n \otimes Y^*, \, \|S^*\| \leq 1 \}.
\end{equation}
Recall that the 2-summing norm $\pi_2$ is self-dual with respect to the trace duality. Therefore, we have
\begin{equation}
\label{eq:pi2_dual}
\mathrm{tr}(S^*T^*) \leq \pi_2(S^*) \pi_2(T^*).
\end{equation}
Since $S^* \in \ell_1^n \otimes Y^*$ and $Y$ is of type-2, by \cite[Propositions~9.3 and 9.8]{jameson1987summing}, it follows that
\begin{equation}
\label{eq:pi2_type2}
\pi_2(S^*) \leq \sqrt{2} \, \tau_2(Y) \|S^*\|.
\end{equation}
By plugging in \eqref{eq:pi2_dual} and \eqref{eq:pi2_type2} into \eqref{eq:pi_dual}, we obtain
\[
\nu_1(T^*) \leq \sqrt{2} \, \tau_2(Y) \pi_2(T^*).
\]
Moreover, since all Banach spaces here are finite-dimensional, it follows from \cite[Proposition~1.13]{jameson1987summing} that
\[
\nu_1(T) = \nu_1(T^*).
\]
The lower bound is obtained from the fact that the nuclear-1 norm is the largest operator ideal norm and the 2-summing norm is an operator ideal norm. 
This completes the proof.

\section{Solving a norm-constrained least squares}
\label{sec:app_least_squares_l2_constraint}
We consider a norm-constrained least squares problem in the form of
\begin{equation}
\label{eq:ncls}
\mathop{\mathrm{minimize}}_{\|x\|_2 \leq \alpha } \ \| y - Ax \|_2^2,
\end{equation}
where $A \in R^{m \times n}$, $y \in R^m$, and $x \in R^n$. 
Then the Lagrangian is given by
\[ 
L(x, \lambda) =  \|y - Ax \|_2^2  + \lambda (\|x\|_2^2 - \alpha^2 ). 
\]
Since \eqref{eq:ncls} satisfies Slater's condition, an optimal solution is characterized by the Karush–Kuhn–Tucker (KKT) conditions. 
The KKT conditions for \eqref{eq:ncls} are given by 
\begin{align*}
    \|x\|_2^2 & \leq \alpha^2, \\
    \lambda & \geq 0, \\
    \lambda (\| x\|_2^2 - \alpha^2 ) & = 0, \\
    ( A^\top A x - A^\top y ) + \lambda x
 & = 0. 
\end{align*}
If the unconstrained least squares solution $x_0 = (A^\top A)^{-1} A^\top y$ is feasible by satisfying the norm constraint $\| x_0 \|_2 \leq \alpha$, then $x_{\mathrm{LS}}$ and $\lambda=0$ satisfy the KKT conditions. 
Therefore, $x_0$ is an optimal solution to \eqref{eq:ncls}. 
Otherwise, we need to find the optimal Lagrange multiplier $\lambda^\star$ that satisfies the KKT conditions. 
Note that $x_\lambda := (A^\top A + \lambda I_n)^{-1} A^\top y$ satisfies the last condition for all $\lambda \geq 0$. Further, the norm of this solution $\|x_\lambda\|_2^2 = \|(A^\top A + \lambda I_n)^{-1} A^\top y\|_2^2$ is a decreasing function of $\lambda$. 
By using a binary search, we can arrive at $\lambda$ that can approximately satisfy all the constraints for any desired level of accuracy. 
At each step in a binary search, $x_\lambda$ needs to be computed. However, this can be done efficiently by pre-computing the SVD of the matrix $A$.

\section{Non-local case analysis}
\label{sec:nonlocal}

Since $A_1,\dots,A_n$ be independent copies of a random matrix with i.i.d. entries following $\mathcal{N}(0,1)$, it follows that 
\begin{equation}
\label{eq:iso_iidg}
\E \sum_{k=1}^n \mathrm{tr}(A_k^\top M)^2 = n \norm{M}_\mathrm{F}^2, \quad \forall M.
\end{equation}
The left-hand side of \eqref{eq:iso_iidg} without the expectation is written as a Gaussian quadratic form given by
\begin{equation}
\label{eq:nonlocal_identity1}
\sum_{k=1}^n \mathrm{tr}(A_k^\top M)^2 = 
\norm{Q_M \xi}_2^2,
\end{equation}
where
\begin{equation*}
Q_{M} := \left( I_n \otimes \mathrm{vec}(M)^\top \right)
\end{equation*} 
and $\xi := [\mathrm{vec}(A_1)\; \dots \; \mathrm{vec}(A_n)] \in \mathbb{R}^{nd_1d_2} \sim \mathcal{N}(0,I_{nd_1d_2})$. 
The identity in \eqref{eq:nonlocal_identity1} holds in the sense of distribution.

Recall that we are interested in reconstructing a low-rank matrix $M$ from the linear measurements in the form of
\[
y_k = \mathrm{tr}(A_k^\top M) + \eta_k, \quad k \in [n], 
\]
where $\eta_1, \dots, \eta_n$ are i.i.d. following $\mathcal{N}(0,\sigma^2)$, which are also independent from everything else. 
To invoke Proposition~\ref{prop:deterministic_errbnd}, we need to derive probabilistic upper bounds in \eqref{eq:qub} and \eqref{eq:gaussian_complexity}. 

We first show that 
\begin{equation}
\label{eq:qub_nonlocal2}
\begin{aligned}
\sup_{\vert\vert\vert M\vert\vert\vert \leq 1} \left| \frac{1}{n} \sum_{k=1}^n \mathrm{tr}(A_k^\top M)^2 - \| M \|_\mathrm{F}^2 \right| 
& \lesssim 
\sqrt{\frac{r(d_1+d_2)\log^6(d_1 d_2)}{n}} \vee \frac{r(d_1+d_2)\log^6(d_1 d_2)}{n} \\
& \quad + \frac{1}{\sqrt{n}}
\left ( \sqrt{\frac{r(d_1+d_2)\log^6(d_1 d_2)}{n}} + 1\right ) \sqrt{\log(2\zeta^{-1})} + \frac{\log(2\zeta^{-1})}{n}
\end{aligned}
\end{equation}
holds with probability $1-\zeta$. 
By \eqref{eq:nonlocal_identity1}, the left-hand side of \eqref{eq:qub_nonlocal2} satisfies
\begin{equation}
\label{eq:qub_nonlocal3}
\sup_{\vert\vert\vert M\vert\vert\vert \leq 1} \left| \frac{1}{n} \sum_{k=1}^n \mathrm{tr}(A_k^\top M)^2 - \| M \|_\mathrm{F}^2 \right| 
=
\sup_{\vert\vert\vert M\vert\vert\vert \leq 1} \left|\norm{Q_M \xi}_2^2 - n \| M \|_\mathrm{F}^2 \right|.
\end{equation}
We apply Theorem~\ref{thm:kmr} to get a tail bound on the right-hand side of \eqref{eq:qub_nonlocal3}. Let $\Delta = \{Q_M : M \in B_{\vert\vert\vert\cdot\vert\vert\vert}\}$.
The radii of $\Delta$ with respect to the Frobenius and spectral norms are given by
\[
d_\mathrm{F}(\Delta) = \sup_{M \in B_{\vert\vert\vert\cdot\vert\vert\vert}} \norm{Q_M}_\mathrm{F} = \sup_{M \in B_{\vert\vert\vert\cdot\vert\vert\vert}} \sqrt{n} \norm{M}_\mathrm{F} \leq \sqrt{n}
\]
and
\[
d_\mathrm{S}(\Delta) = \sup_{M \in B_{\vert\vert\vert\cdot\vert\vert\vert}} \norm{Q_M} = \sup_{M \in B_{\vert\vert\vert\cdot\vert\vert\vert}} \norm{M}_\mathrm{F} \leq 1.
\]
Furthermore, since $\norm{Q_M - Q_{M'}} = \norm{M - M'}_\mathrm{F}$, the $\gamma_2$-functional of $\Delta$ is upper bounded by Dudley's inequality as
\begin{align}
    \tilde{\gamma}_2 (\Delta, \norm{\cdot}) 
    & \lesssim \int_0^{\infty} \sqrt{\log N \left(\Delta, \eta B_\mathrm{S} \right)} \, d\eta \nonumber \\
    & \leq \int_0^{\infty} \sqrt{\log N\left ( B_{\vert\vert\vert\cdot\vert\vert\vert}, \eta B_{S_2^{d_1,d_2}} \right )} \, d\eta \nonumber \\
    & \leq \int_0^{\infty} \sqrt{\log N\left ( \sqrt{r} B_{S_1^{d_1,d_2}}, \eta B_{S_2^{d_1,d_2}} \right )} \, d\eta \nonumber \\  
    & \lesssim \sqrt{r} \mathcal{E}_{2,1}(\mathrm{id}: S_1^{d_1,d_2} \rightarrow S_2^{d_1,d_2}), \label{eq:iidG_gamma2}
\end{align}
where $S_q^{d_1,d_2}$ denotes the Schatten-$q$ class of $d_1$-by-$d_2$ matrices. 
The upper bound in \eqref{eq:iidG_gamma2} is further bounded from above by using Maurey's empirical method as follows. 
Note that the unit ball in $S_1^{d_1,d_2}$ is contained in the unit ball in $S_q^{d_1,d_2}$ for any $q \geq 1$. 
Then we apply \cite[Equation (3)]{junge2020generalized} and \cite[Theorem 3.15]{junge2020generalized} to an upper bound by substituting $S_1^{d_1,d_2}$ by $S_q^{d_1,d_2}$, which yields
\[
\mathcal{E}_{2,1}(\mathrm{id}: S_q^{d_1,d_2} \rightarrow S_2^{d_1,d_2}) \lesssim (1+ \ln d_1 + \ln d_2)^{3/2} (q')^{3/2} \alpha_{d_1 d_2}(\mathrm{id}),
\]
where $q'$ satisfies $1/q + 1/q' = 1$ and 
\[
\alpha_{d_1 d_2}(\mathrm{id}) 
= \left\| \sqrt{d_1} I_{d_2} \right\|_{S_{q'}^{d_2}} + \left\| \sqrt{d_2} I_{d_1} \right\|_{S_{q'}^{d_1}}
= d_1^{1/2} d_2^{1/q'} + d_1^{1/q'} d_2^{1/2}.
\]
By choosing $q' = \log d_1 + \log d_2$, we obtain 
\begin{equation}
\label{eq:iidG_gamma2_E21}
\mathcal{E}_{2,1}(\mathrm{id}: S_1^{d_1,d_2} \rightarrow S_2^{d_1,d_2}) \lesssim \sqrt{d_1+d_2} \log^3(d_1 d_2).
\end{equation}
Then plugging in \eqref{eq:iidG_gamma2_E21} to \eqref{eq:iidG_gamma2} provides
\[
\tilde{\gamma}_2 (\Delta, \norm{\cdot}) 
\lesssim \sqrt{r(d_1+d_2)\log^6(d_1 d_2)}.
\]
Then $E$, $U$, and $V$ in Theorem~\ref{thm:kmr} are upper-bounded respectively by
\begin{align*}
E & \lesssim 
n \left( \sqrt{\frac{r(d_1+d_2)\log^6(d_1 d_2)}{n}} \vee \frac{r(d_1+d_2)\log^6(d_1 d_2)}{n} \right), 
\quad U \leq 1,
\intertext{and}
V & \lesssim \sqrt{n}
\left ( \sqrt{\frac{r(d_1+d_2)\log^6(d_1 d_2)}{n}} \vee 1\right ).
\end{align*}
Therefore, plugging in these upper estimates to Theorem~\ref{thm:kmr} provides \eqref{eq:qub_nonlocal2}. 

Next, we show that 
\begin{equation}
\label{eq:nonlocal_gaussian_complexity}
\E_{(g_k)} \left\vert\xspace\left\vert\xspace\left\vert\sum_{k=1}^n g_{k} A_{k} \right\vert\xspace\right\vert\xspace\right\vert_{*} 
\lesssim \sqrt{nr(d_1+d_2) \log(\zeta^{-1}) \log^6(d_1 d_2)}
\end{equation} 
holds with probability $1-\zeta$, where the probability is with respect to the random matrices $A_1,\dots,A_n$. 
By \citep[Equation (4.9)]{ledoux2013probability}, the left-hand side of \eqref{eq:nonlocal_gaussian_complexity} is upper-bounded by the corresponding Rademacher complexity multiplied by a logarithmic factor, i.e. 
\begin{equation*}
\mathbb{E}_{(g_k)} \tnorm{\sum_{k=1}^n g_k A_k}_*
\lesssim \sqrt{\log n} \cdot
\mathbb{E}_{(r_k)} \tnorm{\sum_{k=1}^n r_k A_k}_*
\end{equation*}
for a Rademacher sequence $(r_k)_{k=1}^n$. Then, due to the symmetry of the standard Gaussian distribution, we obtain
\begin{equation*}
\begin{aligned}
\E_{(r_k)} \tnorm{\sum_{k=1}^n r_k A_k}_{*} 
= \sup_{\vert\vert\vert M \vert\vert\vert \leq 1} \left| \sum_{k=1}^n \langle r_k A_k, M \rangle \right| 
= \underbrace{ \sup_{\vert\vert\vert M \vert\vert\vert \leq 1}  \sum_{k=1}^n \langle A_k, M \rangle }_{\text{($\S$)}}.
\end{aligned},
\end{equation*}
where the last identity used the fact that the set $\{M: \tnorm{M} \leq \alpha\}$ is symmetric.
Below we derive a tail bound on ($\S$) by using Lemma~\ref{lemma:gaussianmax}.
Let $f_M := \mathbf{1}_{n,1} \otimes \mathrm{vec}(M) \in \mathbb{R}^{n d_1 d_2}$, where $\mathbf{1}_{n,1}$ denotes the column vector of length $n$ with all entries set to $1$.
Then ($\S$) is written as the maximum of $f^* \xi$ with $\xi \sim \mathcal{N}(0,I_{nd_1d_2})$ over the set $\Delta = \{ f_M : \tnorm{M} \leq 1 \}$.
Since 
\begin{align*}
\norm{f_M - f_{M'}}_2 
= \sqrt{n} \norm{M - M'}_\mathrm{F}, 
\end{align*}
it follows that
\[
N(\Delta, \eta B_2) 
\leq N\left( B_{\vert\vert\vert\cdot\vert\vert\vert}, \frac{\eta}{\sqrt{n}} B_{S_2^{d_1,d_2}} \right)
\leq N\left( \sqrt{r} B_{S_1^{d_1,d_2}}, \frac{\eta}{\sqrt{n}} B_{S_2^{d_1,d_2}} \right).
\]
Therefore, we have
\begin{align*}
\int_0^{\infty} \sqrt{\log N\left ( \Delta, \frac{\eta}{\sqrt{n}} B_2 \right )} \, d\eta 
& \leq 
\int_0^{\infty} \sqrt{\log N\left ( \sqrt{r} B_{S_1^{d_1,d_2}}, \eta B_{S_2^{d_1,d_2}} \right )} \, d\eta \\
& \lesssim \sqrt{nr} \mathcal{E}_{2,1}(\mathrm{id}: S_1^{d_1,d_2} \rightarrow S_2^{d_1,d_2}) \\
& \lesssim \sqrt{nr(d_1+d_2)} \log^3(d_1 d_2).
\end{align*}
Plugging in this upper bound to Lemma~\ref{lemma:gaussianmax} provides \eqref{eq:nonlocal_gaussian_complexity}. 
Finally, \eqref{eq:samp_comp_nonlocal} is obtained by invoking Proposition~\ref{prop:deterministic_errbnd} with \eqref{eq:qub_nonlocal2} and \eqref{eq:nonlocal_gaussian_complexity}.



\begin{thebibliography}{55}
\providecommand{\natexlab}[1]{#1}
\providecommand{\url}[1]{\texttt{#1}}
\expandafter\ifx\csname urlstyle\endcsname\relax
  \providecommand{\doi}[1]{doi: #1}\else
  \providecommand{\doi}{doi: \begingroup \urlstyle{rm}\Url}\fi

\bibitem[Anaraki and Hughes(2014)]{anaraki2014memory}
Farhad~Pourkamali Anaraki and Shannon Hughes.
\newblock Memory and computation efficient {PCA} via very sparse random
  projections.
\newblock In \emph{International Conference on Machine Learning}, pages
  1341--1349. PMLR, 2014.

\bibitem[Azizyan et~al.(2018)Azizyan, Krishnamurthy, and
  Singh]{azizyan2018extreme}
Martin Azizyan, Akshay Krishnamurthy, and Aarti Singh.
\newblock Extreme compressive sampling for covariance estimation.
\newblock \emph{IEEE Transactions on Information Theory}, 64\penalty0
  (12):\penalty0 7613--7635, 2018.

\bibitem[Boyd et~al.(2011)Boyd, Parikh, Chu, Peleato, and
  Eckstein]{boyd2011distributed}
Stephen Boyd, Neal Parikh, Eric Chu, Borja Peleato, and Jonathan Eckstein.
\newblock Distributed optimization and statistical learning via the alternating
  direction method of multipliers.
\newblock \emph{Foundations and Trends{\textregistered} in Machine learning},
  3\penalty0 (1):\penalty0 1--122, 2011.

\bibitem[Bruer(2017)]{bruer2017recovering}
John~Jacob Bruer.
\newblock \emph{Recovering structured low-rank operators using nuclear norms}.
\newblock PhD thesis, California Institute of Technology, 2017.

\bibitem[Cai and Zhou(2016)]{cai2016matrix}
T~Tony Cai and Wen-Xin Zhou.
\newblock Matrix completion via max-norm constrained optimization.
\newblock \emph{Electronic Journal of Statistics}, 10\penalty0 (1):\penalty0
  1493--1525, 2016.

\bibitem[Candes and Recht(2012)]{candes2012exact}
Emmanuel Candes and Benjamin Recht.
\newblock Exact matrix completion via convex optimization.
\newblock \emph{Communications of the ACM}, 55\penalty0 (6):\penalty0 111--119,
  2012.

\bibitem[Candes and Plan(2011)]{candes2011tight}
Emmanuel~J Candes and Yaniv Plan.
\newblock Tight oracle inequalities for low-rank matrix recovery from a minimal
  number of noisy random measurements.
\newblock \emph{IEEE Transactions on Information Theory}, 57\penalty0
  (4):\penalty0 2342--2359, 2011.

\bibitem[Carl and Stephani(1990)]{carl1990entropy}
B.~Carl and I.~Stephani.
\newblock \emph{Entropy, Compactness and the Approximation of Operators}.
\newblock Cambridge Tracts in Mathematics. Cambridge University Press, 1990.

\bibitem[Carl and Stephani(2008)]{carl2008entropy}
B.~Carl and I.~Stephani.
\newblock \emph{Entropy, Compactness and the Approximation of Operators}.
\newblock Cambridge Tracts in Mathematics. Cambridge University Press, 2008.

\bibitem[Carl(1985)]{carl1985inequalities}
Bernd Carl.
\newblock Inequalities of {B}ernstein-{J}ackson-type and the degree of
  compactness of operators in {B}anach spaces.
\newblock \emph{Ann. Inst. Fourier (Grenoble)}, 35\penalty0 (3):\penalty0
  79--118, 1985.

\bibitem[Chen et~al.(2016)Chen, He, Ye, and Yuan]{chen2016direct}
Caihua Chen, Bingsheng He, Yinyu Ye, and Xiaoming Yuan.
\newblock The direct extension of {ADMM} for multi-block convex minimization
  problems is not necessarily convergent.
\newblock \emph{Mathematical Programming}, 155\penalty0 (1-2):\penalty0 57--79,
  2016.

\bibitem[Cover and Thomas(1991)]{cover1991elements}
T.M. Cover and J.A. Thomas.
\newblock \emph{Elements of Information Theory}.
\newblock Wiley Series in Telecommunications and Signal Processing. Wiley,
  1991.

\bibitem[Defant and Floret(1992)]{defant1992tensor}
Andreas Defant and Klaus Floret.
\newblock \emph{Tensor norms and operator ideals}, volume 176.
\newblock Elsevier, 1992.

\bibitem[Diestel et~al.(2008)Diestel, Grothendieck, Fourie, and
  Swart]{diestel2008metric}
J.~Diestel, A.~Grothendieck, J.H. Fourie, and J.~Swart.
\newblock \emph{The Metric Theory of Tensor Products: Grothendieck's
  R{\'e}sum{\'e} Revisited}.
\newblock Amsns AMS non-series Title Series. American Mathematical Soc., 2008.
\newblock ISBN 9780821872697.

\bibitem[Diestel(1985)]{diestel1985introduction}
Joe Diestel.
\newblock An introduction to the theory of absolutely p-summing operators
  between banach spaces.
\newblock In \emph{Miniconference on Linear Analysis and Functional Spaces},
  volume~9, pages 1--27. Australian National University, Mathematical Sciences
  Institute, 1985.

\bibitem[Dirksen(2015)]{dirksen2015tail}
Sjoerd Dirksen.
\newblock Tail bounds via generic chaining.
\newblock \emph{Electronic Journal of Probability}, 20, 2015.

\bibitem[Fang et~al.(2018)Fang, Liu, Toh, and Zhou]{fang2018max}
E.~X. Fang, H.~Liu, K.~Toh, and W.~Zhou.
\newblock Max-norm optimization for robust matrix recovery.
\newblock \emph{Mathematical Programming}, 167\penalty0 (1):\penalty0 5--35,
  2018.

\bibitem[Figiel et~al.(1977)Figiel, Lindenstrauss, and
  Milman]{figiel1977dimension}
Tadeusz Figiel, Joram Lindenstrauss, and Vitali~D Milman.
\newblock The dimension of almost spherical sections of convex bodies.
\newblock \emph{Acta Mathematica}, 139\penalty0 (1):\penalty0 53--94, 1977.

\bibitem[Foucart and Rauhut(2013)]{foucart2013mathematical}
S.~Foucart and H.~Rauhut.
\newblock \emph{A Mathematical Introduction to Compressive Sensing}.
\newblock Applied and Numerical Harmonic Analysis. Springer New York, 2013.

\bibitem[Foygel and Srebro(2011)]{foygel2011concentration}
Rina Foygel and Nathan Srebro.
\newblock Concentration-based guarantees for low-rank matrix reconstruction.
\newblock In \emph{Proceedings of the 24th Annual Conference on Learning
  Theory}, pages 315--340. JMLR Workshop and Conference Proceedings, 2011.

\bibitem[Friedland and Gu{\'e}don(2011)]{friedland2011random}
Omer Friedland and Olivier Gu{\'e}don.
\newblock Random embedding of $\ell^n_p$ into $\ell^n_r$.
\newblock \emph{Math. Ann}, 350\penalty0 (4):\penalty0 953--972, 2011.

\bibitem[Giladi et~al.(2017)Giladi, Prochno, Sch{\"u}tt, Tomczak-Jaegermann,
  and Werner]{giladi2017geometry}
Ohad Giladi, Joscha Prochno, Carsten Sch{\"u}tt, Nicole Tomczak-Jaegermann, and
  Elisabeth Werner.
\newblock On the geometry of projective tensor products.
\newblock \emph{Journal of Functional Analysis}, 273\penalty0 (2):\penalty0
  471--495, 2017.

\bibitem[Grothendieck(1956)]{grothendieck1956resume}
Alexandre Grothendieck.
\newblock \emph{R{\'e}sum{\'e} de la th{\'e}orie m{\'e}trique des produits
  tensoriels topologiques}.
\newblock Soc. de Matem{\'a}tica de S{\~a}o Paulo, 1956.

\bibitem[Gu{\'e}don et~al.(2007)Gu{\'e}don, Mendelson, Pajor, and
  Tomczak-Jaegermann]{guedon2007subspaces}
Olivier Gu{\'e}don, Shahar Mendelson, Alain Pajor, and Nicole
  Tomczak-Jaegermann.
\newblock Subspaces and orthogonal decompositions generated by bounded
  orthogonal systems.
\newblock \emph{Positivity}, 11\penalty0 (2):\penalty0 269--283, 2007.

\bibitem[Jameson(1987)]{jameson1987summing}
Graham James~Oscar Jameson.
\newblock \emph{Summing anefant nuclear norms in Banach space theory},
  volume~8.
\newblock Cambridge University Press, 1987.

\bibitem[Johnson and Schechtman(1982)]{johnson1982embeddingl}
William~B Johnson and Gideon Schechtman.
\newblock Embedding $\ell_p^m$ into $\ell_1^n$.
\newblock \emph{Acta Mathematica}, 149\penalty0 (1):\penalty0 71--85, 1982.

\bibitem[Junge and Lee(2020)]{junge2020generalized}
Marius Junge and Kiryung Lee.
\newblock Generalized notions of sparsity and restricted isometry property.
  part {I}: a unified framework.
\newblock \emph{Information and Inference: A Journal of the IMA}, 9\penalty0
  (1):\penalty0 157--193, 2020.

\bibitem[Kashin(1977)]{kashin1977diameters}
Boris~Sergeevich Kashin.
\newblock Diameters of some finite-dimensional sets and classes of smooth
  functions.
\newblock \emph{Izvestiya Rossiiskoi Akademii Nauk. Seriya Matematicheskaya},
  41\penalty0 (2):\penalty0 334--351, 1977.

\bibitem[Kliesch et~al.(2016)Kliesch, Kueng, Eisert, and
  Gross]{kliesch2016improving}
Martin Kliesch, Richard Kueng, Jens Eisert, and David Gross.
\newblock Improving compressed sensing with the diamond norm.
\newblock \emph{IEEE Transactions on Information Theory}, 62\penalty0
  (12):\penalty0 7445--7463, 2016.

\bibitem[Krahmer et~al.(2014)Krahmer, Mendelson, and
  Rauhut]{krahmer2014suprema}
Felix Krahmer, Shahar Mendelson, and Holger Rauhut.
\newblock Suprema of chaos processes and the restricted isometry property.
\newblock \emph{Communications on Pure and Applied Mathematics}, 67\penalty0
  (11):\penalty0 1877--1904, 2014.

\bibitem[Ledoux and Talagrand(2013)]{ledoux2013probability}
M.~Ledoux and M.~Talagrand.
\newblock \emph{Probability in Banach Spaces: isoperimetry and processes}.
\newblock Springer Science \& Business Media, 2013.

\bibitem[Lee et~al.(2019)Lee, Srinivasa, Junge, and Romberg]{lee2019entropy}
Kiryung Lee, Rakshith~Sharma Srinivasa, Marius Junge, and Justin Romberg.
\newblock Entropy estimates on tensor products of banach spaces and
  applications to low-rank recovery.
\newblock In \emph{2019 13th International conference on Sampling Theory and
  Applications (SampTA)}, pages 1--4. IEEE, 2019.

\bibitem[Linial et~al.(2007)Linial, Mendelson, Schechtman, and
  Shraibman]{linial2007complexity}
Nati Linial, Shahar Mendelson, Gideon Schechtman, and Adi Shraibman.
\newblock Complexity measures of sign matrices.
\newblock \emph{Combinatorica}, 27\penalty0 (4):\penalty0 439--463, 2007.

\bibitem[Litvak et~al.(2005)Litvak, Pajor, Rudelson, Tomczak-Jaegermann, and
  Vershynin]{litvak2005euclidean}
AE~Litvak, A~Pajor, M~Rudelson, Nicole Tomczak-Jaegermann, and R~Vershynin.
\newblock Euclidean embeddings in spaces of finite volume ratio via random
  matrices.
\newblock \emph{Journal f{\"u}r die reine und angewandte Mathematik},
  2005\penalty0 (589):\penalty0 1--19, 2005.

\bibitem[Lofberg(2004)]{lofberg2004yalmip}
Johan Lofberg.
\newblock {YALMIP}: A toolbox for modeling and optimization in matlab.
\newblock In \emph{2004 IEEE international conference on robotics and
  automation (IEEE Cat. No. 04CH37508)}, pages 284--289. IEEE, 2004.

\bibitem[Nayer and Vaswani(2021)]{nayer2021fast}
Seyedehsara Nayer and Namrata Vaswani.
\newblock Fast and sample-efficient federated low rank matrix recovery from
  column-wise linear and quadratic projections.
\newblock \emph{arXiv preprint arXiv:2102.10217}, 2021.

\bibitem[Nayer et~al.(2019)Nayer, Narayanamurthy, and
  Vaswani]{nayer2019phaseless}
Seyedehsara Nayer, Praneeth Narayanamurthy, and Namrata Vaswani.
\newblock Phaseless pca: Low-rank matrix recovery from column-wise phaseless
  measurements.
\newblock In \emph{International Conference on Machine Learning}, pages
  4762--4770. PMLR, 2019.

\bibitem[Negahban and Wainwright(2011)]{negahban2011estimation}
Sahand Negahban and Martin~J Wainwright.
\newblock Estimation of (near) low-rank matrices with noise and
  high-dimensional scaling.
\newblock \emph{The Annals of Statistics}, pages 1069--1097, 2011.

\bibitem[Negahban and Wainwright(2012)]{negahban2012restricted}
Sahand Negahban and Martin~J Wainwright.
\newblock Restricted strong convexity and weighted matrix completion: Optimal
  bounds with noise.
\newblock \emph{The Journal of Machine Learning Research}, 13\penalty0
  (1):\penalty0 1665--1697, 2012.

\bibitem[Pisier(1983)]{pisier1983dimension}
Gilles Pisier.
\newblock On the dimension of the $\ell^n_p$-subspaces of {B}anach spaces, for
  $1 \leq p < 2$.
\newblock \emph{Transactions of the American Mathematical Society},
  276\penalty0 (1):\penalty0 201--211, 1983.

\bibitem[Pisier(1999)]{pisier1999volume}
Gilles Pisier.
\newblock \emph{The volume of convex bodies and {B}anach space geometry},
  volume~94.
\newblock Cambridge University Press, 1999.

\bibitem[Pisier(2012)]{pisier2012grothendieck}
Gilles Pisier.
\newblock Grothendieck’s theorem, past and present.
\newblock \emph{Bulletin of the American Mathematical Society}, 49\penalty0
  (2):\penalty0 237--323, 2012.

\bibitem[Pisier et~al.(1986)]{pisier1986factorization}
Gilles Pisier et~al.
\newblock \emph{Factorization of linear operators and geometry of Banach
  spaces}, volume~60.
\newblock American Mathematical Soc., 1986.

\bibitem[Rauhut(2010)]{rauhut2010compressive}
Holger Rauhut.
\newblock Compressive sensing and structured random matrices.
\newblock \emph{Theoretical foundations and numerical methods for sparse
  recovery}, 9:\penalty0 1--92, 2010.

\bibitem[Rudelson and Vershynin(2008)]{rudelson2008sparse}
Mark Rudelson and Roman Vershynin.
\newblock On sparse reconstruction from fourier and gaussian measurements.
\newblock \emph{Communications on Pure and Applied Mathematics}, 61\penalty0
  (8):\penalty0 1025--1045, 2008.

\bibitem[Srebro and Shraibman(2005)]{srebro2005rank}
Nathan Srebro and Adi Shraibman.
\newblock Rank, trace-norm and max-norm.
\newblock In \emph{International Conference on Computational Learning Theory},
  pages 545--560. Springer, 2005.

\bibitem[Srebro et~al.(2004)Srebro, Rennie, and Jaakkola]{srebro2004maximum}
Nathan Srebro, Jason D.~M. Rennie, and Tommi~S. Jaakkola.
\newblock Maximum-margin matrix factorization.
\newblock In \emph{Proceedings of the 17th International Conference on Neural
  Information Processing Systems}, page 1329–1336, Cambridge, MA, USA, 2004.
  MIT Press.

\bibitem[Srinivasa et~al.(2019)Srinivasa, Lee, Junge, and
  Romberg]{srinivasa2019decentralized}
Rakshith~Sharma Srinivasa, Kiryung Lee, Marius Junge, and Justin Romberg.
\newblock Decentralized sketching of low rank matrices.
\newblock \emph{Advances in Neural Information Processing Systems},
  32:\penalty0 10101--10110, 2019.

\bibitem[Sturm(1999)]{sturm1999sedumi}
Jos~F. Sturm.
\newblock Using {SeDuMi} 1.02, a {MATLAB} toolbox for optimization over
  symmetric cones.
\newblock \emph{Optimization Methods and Software}, 11\penalty0 (1-4):\penalty0
  625--653, 1999.
\newblock \doi{10.1080/10556789908805766}.

\bibitem[Szarek(1978)]{szarek1978kashins}
Stanis{\l}aw~Jerzy Szarek.
\newblock Kashin's almost {E}uclidean orthogonal decomposition of $\ell_n^1$.
\newblock \emph{Bulletin de L'Acad\'emie Polonaise Des Sciences: S\'erie des
  sciences math\'ematiques, astronomiques, et physiques}, 26\penalty0
  (8):\penalty0 691--694, 1978.

\bibitem[Szarek and Tomczak-Jaegermann(1980)]{szarek1980nearly}
Stanis{\l}law Szarek and Nicole Tomczak-Jaegermann.
\newblock On nearly euclidean decomposition for some classes of banach spaces.
\newblock \emph{Compositio Mathematica}, 40\penalty0 (3):\penalty0 367--385,
  1980.

\bibitem[Talagrand(1998)]{talagrand1998selecting}
Michel Talagrand.
\newblock Selecting a proportion of characters.
\newblock \emph{Israel Journal of Mathematics}, 108\penalty0 (1):\penalty0 173,
  1998.

\bibitem[Tibshirani(1996)]{tibshirani1996regression}
Robert Tibshirani.
\newblock Regression shrinkage and selection via the lasso.
\newblock \emph{Journal of the Royal Statistical Society: Series B
  (Methodological)}, 58\penalty0 (1):\penalty0 267--288, 1996.

\bibitem[Tong et~al.(2020)Tong, Ma, and Chi]{tong2020accelerating}
Tian Tong, Cong Ma, and Yuejie Chi.
\newblock Accelerating ill-conditioned low-rank matrix estimation via scaled
  gradient descent.
\newblock \emph{arXiv preprint arXiv:2005.08898}, 2020.

\bibitem[Vershynin(2018)]{vershynin2018high}
Roman Vershynin.
\newblock \emph{High-Dimensional Probability: An Introduction with Applications
  in Data Science}, volume~47.
\newblock Cambridge University Press, 2018.

\end{thebibliography}
%

\end{document}